\documentclass{article}

\RequirePackage[OT1]{fontenc}
\RequirePackage{amsthm,amsmath,amssymb}
\RequirePackage[numbers,square]{natbib}
\RequirePackage[colorlinks]{hyperref}

\usepackage{bbm}

\newcommand*\Bell{\ensuremath{\boldsymbol\ell}}

\theoremstyle{definition}
\newtheorem{theorem}{Theorem}[section] %% numbering on section level

\newtheorem{remark}[theorem]{Remark} %% dito
\newtheorem{definition}[theorem]{Definition} %[section]
\newtheorem{example}[theorem]{Example} %[section] % [section]
\newtheorem{lemma}[theorem]{Lemma} %[section]
\newtheorem{conjecture}[theorem]{Conjecture} %[section]

\newcommand{\set}[1]{\{#1\}}
\newcommand{\divides}[2]{#1\mid#2}

\newcommand{\real}{\mathbb{R}}
\newcommand{\nn}{\mathbb{N}}
\newcommand{\sd}{\,|\,} % ``so dass''
\newcommand{\goesto}{\rightarrow} % sequences 
\DeclareMathOperator{\Exp}{E} % expected value --> stochastic
\DeclareMathOperator{\Var}{Var} % variance --> stochastic
\DeclareMathOperator{\Cov}{Cov} % covariance --> stochastic
\begin{document}

\title{The Combinatorics of Weighted Vector Compositions}

%\begin{aug}
%\author{\fnms{Steffen}
%\snm{Eger}\ead[label=e1]{eger@ukp.informatik.tu-darmstadt.de}}
%\address{Hochschulstrasse 10\\
%64289 Darmstadt, Germany
%}
%\printead{e1}}

%\affiliation{Technical University Darmstadt}
%\end{aug}

%\author{Steffen Eger\\
%Hochschulstrasse 10\\
%64289 Darmstadt, Germany\\
%{\it eger@ukp.informatik.tu-darmstadt.de}
%}
\author{Steffen Eger\\
Department of Computer Science\\
Technical University Darmstadt \\
%Hochschulstrasse 10\\
%64289 Darmstadt, Germany\\
{\it eger@ukp.informatik.tu-darmstadt.de}
}

\date{}
\maketitle

\begin{abstract}
A vector composition of a vector $\Bell$ is a matrix $\mathbf{A}$ whose rows sum to
$\Bell$. We define a weighted vector composition as a vector
composition in which the column values of $\mathbf{A}$ may appear in
different colors. 
We study vector compositions from
different viewpoints: (1) We show how they are related to sums of random
vectors and (2) how they allow to derive formulas for partial
derivatives of composite functions. (3) We 
study congruence properties of the number of weighted vector
compositions, for fixed and arbitrary number of parts,
many of which are analogous to 
those of ordinary binomial coefficients and related quantities. Via the
Central Limit Theorem and their multivariate generating functions, 
(4) we also investigate the asymptotic behavior of several special cases
of numbers of weighted vector compositions. Finally, (5) we conjecture
an extension 
of a primality criterion due to Mann and Shanks \cite{Mann:1972} in
the context of weighted vector compositions. 
\end{abstract}

\section{Introduction}
An integer composition (ordered partition) of a nonnegative integer $n$ is a tuple
$(\pi_1,\ldots,\pi_k)$ of nonnegative integers whose sum is $n$. The
$\pi_i$'s are 
called the \emph{parts} of the
composition. For fixed number $k$ of parts, the number of
\emph{$f$-weighted} 
integer compositions---also called \emph{$f$-colored integer
  compositions} in
the literature---in which each part size $s$ may occur in $f(s)$
different colors, is given by the \emph{extended binomial coefficient}
$\binom{k}{n}_{f}$ \cite{Eger:2013}. 

We generalize here the notion of weighted integer
compositions to \emph{weighted vector compositions}. 
For a vector
$\Bell\in\nn^N$, for $N\ge 1$, a vector composition
\cite{Andrews:1975} of $\Bell$ with
$k$ \emph{parts} is a matrix
$\mathbf{A}=[\mathbf{m}_1,\ldots,\mathbf{m}_k]\in \nn^{N\times k}$
such that $\mathbf{m}_1+\cdots+\mathbf{m}_k=\Bell$. We call a vector
composition \emph{$f$-weighted}, for a function $f:\nn^N\rightarrow \nn$,
when each part of `size' $\mathbf{m}$ may occur in one of $f(\mathbf{m})$
different colors in the composition. For example, for $N=2$ and $f$: 
\begin{align*}
  f\bigl((1,1)\bigr)=2, \quad f\bigl((1,0)\bigr)=1,\quad f\bigl((0,1)\bigr)=1 
\end{align*}
and $f(\mathbf{x})=0$ for all other $\mathbf{x}\in\nn^2$, there are
seven distinct $f$-weighted vector
compositions of $\Bell=(1,2)$, namely:
\begin{align*}
  &\left[\begin{pmatrix} 0 \\ 1\end{pmatrix}
    \begin{pmatrix}1 \\ 1\end{pmatrix}
    \right],\:\:
  \left[\begin{pmatrix} 1 \\ 1\end{pmatrix}
    \begin{pmatrix}0 \\ 1\end{pmatrix}
    \right],\:\:
  \left[\begin{pmatrix} 0 \\ 1\end{pmatrix}
    \begin{pmatrix}1 \\ 1\end{pmatrix}^{\Diamond}
    \right],\:\:
  \left[\begin{pmatrix} 1 \\ 1\end{pmatrix}^{\Diamond}
    \begin{pmatrix}0 \\ 1\end{pmatrix}
    \right],\\
   & \left[\begin{pmatrix} 0 \\ 1\end{pmatrix}
    \begin{pmatrix}0 \\ 1\end{pmatrix}
      \begin{pmatrix}1 \\ 0\end{pmatrix}
    \right],\:\:
   \left[\begin{pmatrix} 0 \\ 1\end{pmatrix}
    \begin{pmatrix}1 \\ 0\end{pmatrix}
      \begin{pmatrix}0 \\ 1\end{pmatrix}
    \right],\:\:
   \left[\begin{pmatrix} 1 \\ 0\end{pmatrix}
    \begin{pmatrix}0 \\ 1\end{pmatrix}
      \begin{pmatrix}0 \\ 1\end{pmatrix}
    \right]
\end{align*}
where $\Diamond$ distinguishes between the two values of $(1,1)$. 
For fixed number $k\ge 0$ of parts, we denote the number of distinct
$f$-weighted vector compositions of $\Bell\in\nn^N$ by
$\binom{k}{\Bell}_f$. Moreover, the number $c_f(\Bell)$ of
$f$-weighted vector compositions with arbitrarily many parts is then
given by $c_f(\Bell)=\sum_{k\ge 0}\binom{k}{\Bell}_f$. 

The number of $f$-weighted vector compositions with $k$ parts may be
represented as
\begin{align}\label{eq:repr1}
  \binom{k}{\Bell}_f =
  \sum_{\mathbf{m}_1+\cdots+\mathbf{m}_k=\Bell}f(\mathbf{m}_1)\cdots f(\mathbf{m}_k).
\end{align}
When the function $f$ takes values in $\real$ (or even in a
commutative ring), then the RHS of Eq.~(\ref{eq:repr}) gives the
\emph{total weight} of all vector compositions of $\Bell$ with $k$
parts, where we define the \emph{weight} {of a composition}
$[\mathbf{m}_1,\ldots,\mathbf{m}_k]$ as $f(\mathbf{m}_1)\cdots
f(\mathbf{m}_k)$.  

We study $f$-weighted vector compositions from several viewpoints. 
Section \ref{sec:rv} relates
weighted vector compositions to sums of random vectors. Section
\ref{sec:prelim} introduces basic identities for $\binom{k}{\Bell}_f$
which will be used in  
follow-up results. 
Section \ref{sec:partial} derives a formula for partial derivatives of
composite functions using
these identities. Our formula generalizes the famous formula of Fa\`{a} di
Bruno (see \cite{Johnson:2002}) for the higher order derivatives of a composite
function. 
Section \ref{sec:congruence1} gives
divisibility properties of $\binom{k}{\Bell}_f$ and in Section
\ref{sec:congruence2}, we derive congruences and identities for sums of
$\binom{k}{\Bell}_f$, including $c_f(\Bell)$. 
Our results in these two sections generalize corresponding results from \cite{Eger:2016,Shapcott:2013} for weighted integer compositions, and others for ordinary binomial coefficients. We also generalize here the notion of so-called $s$-color compositions in which a part of size $s$ may occur in $s$ different colors in a composition \cite{Agarwal:2000}. 
We discuss 
asymptotics of weighted vector
compositions in Section \ref{sec:asymp} and %we discuss 
the primality
criterion of Mann and Shanks \cite{Mann:1972} in the context of weighted vector
compositions in Section \ref{sec:prime}. 

In the rest of this work, we use the following \textbf{notation} and \textbf{definitions}. 
We write vectors and matrices in bold font ($\mathbf{x},
\Bell,\ldots$) to distinguish them from `scalars' ($k,n,\ldots$). We
write vectors as row vectors ($(x,y,z),\ldots$). We write the
components of a vector $\mathbf{x}$ as $x_1,x_2,\ldots$ and similarly
for matrices. We use the
standard notation, $\binom{k}{n}$, for ordinary binomial coefficients,
which are a special case of our setup. They are retrieved when $N=1$
and $f(x)$ is the \emph{indicator function} on $\set{0,1}$, that is,
$f(x)=1$ for $x\in\set{0,1}$ and $f(x)=0$ for all other $x$. 

We let $\nn=\set{0,1,2,\ldots}$ be the set of nonnegative integers. Let $k\ge 0,N\ge 1$ and let $\Bell\in\nn^N$. Let $\mathbf{0}=\mathbf{0}_N=(0,\ldots,0)\in\nn^N$ and let $\mathbf{1}=\mathbf{1}_N=(1,\ldots,1)\in\nn^N$. 
Let 
\begin{align*}
  S(\Bell)=\set{\mathbf{s}\in\nn^N\sd \mathbf{s}\neq\mathbf{0},0\le
    {s}_j\le \ell_j, j=1,\ldots,N}
\end{align*}
be the set of all non-zero part sizes in $\nn^N$ `bounded from above' by
$\Bell$, and let $S_{\mathbf{0}}(\Bell)=S(\Bell)\cup\set{\mathbf{0}}$. Let the elements in $S(\Bell)$ or $S_{\mathbf{0}}(\Bell)$ be enumerated as $\mathbf{s}_1,\mathbf{s}_2,\ldots$. 
We denote by $\mathcal{P}^{(\mathbf{S_{\mathbf{0}}(\Bell)})}(\Bell;k)=\mathcal{P}(\Bell;k)$ the set
\begin{align*}
  \mathcal{P}(\Bell;k) = \set{\bigl(r_1,r_2,\ldots\bigr)\sd r_i\ge 0,\sum_{i\ge 1}r_i=k,\sum_{i\ge 1}r_i\mathbf{s}_i=\Bell}
\end{align*}
%denote 
%the set 
of \emph{vector partitions} (unordered compositions) of $\Bell$ with $k$ parts, including part size 
$\mathbf{0}$. Here, $r_1,r_2,\ldots$ are the \emph{multiplicities} of the
part sizes $\mathbf{s}_1,\mathbf{s}_2,\ldots$. We similarly define
$\mathcal{P}^{(\mathbf{S(\Bell)})}(\Bell;k)$ as the set
of vector partitions of $\Bell$ with $k$ parts, excluding
$\mathbf{0}$. We write
$\mathcal{P}^{(\mathbf{S_{\mathbf{0}}(\Bell)})}(\Bell)=\mathcal{P}(\Bell)$
for the set of vector partitions, part size $\mathbf{0}$ included, of $\Bell$
with arbitrary number of parts:
\begin{align*}
  \mathcal{P}(\Bell) = \set{\bigl(r_1,r_2,\ldots\bigr)\sd r_i\ge 0,\sum_{i\ge 1}r_i\mathbf{s}_i=\Bell}
\end{align*}
and analogously for $\mathcal{P}^{(\mathbf{S(\Bell)})}(\Bell)$.\footnote{When it is clear from context whether $\mathbf{0}$ is included or not, we may also write $\mathcal{P}(\Bell)$ for both $\mathcal{P}^{(\mathbf{S(\Bell)})}(\Bell)$ and $\mathcal{P}^{(\mathbf{S_{\mathbf{0}}(\Bell)})}(\Bell)$, and similarly for related quantities.} For a
scalar $a$ and a vector $\mathbf{b}$, we write $a|\mathbf{b}$, when
$a|b_i$ for all components $b_i$ of $\mathbf{b}$. 

{\textbf{Background}:}
Weighted vector compositions
generalize the concept of {vector compositions} introduced in
Andrews \cite{Andrews:1975}. In fact, vector compositions are
$f$-weighted vector compositions for which $f=f_{S_0}$ is the {indicator
function} on  
$S_0=\mathbb{N}^N-\set{\mathbf{0}_N}$.
For the 
same $f_{S_0}$, Munarini et 
al.\ \cite{Munarini:2008} introduce \emph{matrix compositions}. These
are matrices whose entries sum to a positive integer $n$ and whose columns are
non-zero. We find that the number $c^{(N)}(n)$ of matrix compositions of $n$ for matrices 
with $N$ rows satisfies
$c^{(N)}(n)=\sum_{\ell_1+\cdots+\ell_N=n}c_{f_{S_0}}(\ell_1,\ldots,\ell_N)$. 

Vector compositions are also closely
related to \emph{lattice path combinatorics} \cite{Stanley:2012}. Lattice paths
are paths from the origin $\mathbf{0}_N$ to some point 
$\Bell=(\ell_1,\ldots,\ell_N)\in\nn^N$ where each step lies in some set $S$. In our
case, each coordinate of each step $\mathbf{s}\in S$ is
nonnegative. Vector compositions %are 
also %closely related to the
generalize the 
concept of \emph{alignments} considered in computational biology and
computational linguistics \cite{Griggs:1990}. For example, the number of
(standard) alignments of $N$ sequences of lengths
$\Bell=(\ell_1,\ldots,\ell_N)$ is given by $c_{f_{S_1}}(\Bell)$, where $f_{S_1}$ is the
indicator function on $S_1=\set{(s_1,\ldots,s_N)\sd s_i\in\set{0,1}}-\set{\mathbf{0}_N}$. When $f_{S}$ is the indicator
function on more `complex' $S\subseteq\nn^N$, %such alignments have been
%referred to as 
$c_{f_S}$ counts 
``many-to-many'' alignments \cite{Eger:2016_align}.   

Weighted integer compositions, that is, the case when $N=1$, go back
to \cite{Moser:1961} and \cite{Hoggatt:1968,Hoggatt:1969}. Recently, they have attracted attention in the
form of so-called \emph{$s$-color compositions}, for which $f$ is
specified as identity function, that is, $f(s)=s$ \cite{Agarwal:2000,Guo:2014,Narang:2008,Agarwal:2016,Shapcott:2013}. 
More general $f$ have been considered in
\cite{Abrate:2014,Birmajer:2016,Eger:2013,Eger:2015,Eger:2016,Janjic:2016,Shapcott:2012}, to
name just a few.  
Results on standard
integer compositions, i.e., where $f$ is the indicator function on 
$\nn-\set{0}$ or a subset thereof, 
are found in \cite{Heubach:2004}.  

\section{Relation to multivariate random variables}\label{sec:rv}
Let $X_1,X_2,\ldots$ be i.i.d.\ discrete random vectors with common
distribution function $f(\mathbf{x})=P[X=\mathbf{x}]$, for
$\mathbf{x}\in\nn^N$. Then the distribution of the sum
$X_1+X_2+\cdots+X_k$ is given by
\begin{align*}
  P[X_1+\cdots+X_k=\Bell] &=
  \sum_{\mathbf{m}_1+\cdots+\mathbf{m}_k=\Bell}P[X_1=\mathbf{m}_1]\cdots
  P[X_k=\mathbf{m}_k] \\
  &= 
  \sum_{\mathbf{m}_1+\cdots+\mathbf{m}_k=\Bell}f(\mathbf{m}_1)\cdots
  f(\mathbf{m}_k) 
  = \binom{k}{\Bell}_f.
\end{align*}
Let $f$ be the discrete uniform measure on some $S\subseteq
\nn^N$. Then 
\begin{align*}
  P[X_1+\cdots+X_k=\Bell] = \left(\frac{1}{|S|}\right)^k\binom{k}{\Bell}_{g_S}
\end{align*} 
where $g_S$ is the {indicator function} on $S$.
Thus
$\binom{k}{\Bell}_{g_S}=|S|^kP[X_1+\cdots+X_k=\Bell]$. Moreover,
$P[X_1+\cdots+X_k=\Bell]$ may be approximated by the multivariate normal
distribution according to the multivariate Central Limit Theorem
(CLT). That is, for large $k$, $P[X_1+\cdots+X_k=\Bell]$ can be approximated by the density 
\begin{align*}
   %\sim
  (2\pi)^{-N/2}|\mathbf{\Sigma}_k|^{-1/2}\exp\left(-\frac{1}{2}(\Bell-\boldsymbol{\mu}_k)^\intercal\mathbf{\Sigma}_k^{-1}(\Bell-\boldsymbol{\mu}_k)\right)
\end{align*}
where $\boldsymbol{\mu}_k=k\boldsymbol{\mu}$ and $\mathbf{\Sigma}_k=k\mathbf{\Sigma}$ are the mean vector
and covariance matrix %(amongst the components of each vector $X_i$) 
of
$X_1+\cdots+X_k$, respectively. 
Here, $\boldsymbol{\mu}$ is the mean vector of each $X_i$ and $\mathbf{\Sigma}$ is the covariance matrix among the components of $X_i$, where $|\mathbf{\Sigma}|$ denotes its determinant. 
The 
approximation holds for large $k$. 

\begin{example}
Let $S=\prod_{j=1}^N\set{0,1,\ldots,\nu_j}$, for integers $\nu_j>0$. 
Let $X_i$ be uniformly distributed on $S$, for all $i=1,\ldots,k$.
We have
$\boldsymbol{\mu}=\Exp[X_i]=(\nu_1/2,\ldots,\nu_N/2)$. Since the components of $X_i$ are independent of each other and since the variance of each component $j$ of $X_i$ is given by $\frac{(\nu_j+1)^2-1}{12}$ (variance of uniform distributed random variable on $\set{0,\ldots,\nu_j}$), we find that
\begin{align*}
|\mathbf{\Sigma}|=%\left(
\prod_{j=1}^N\frac{(\nu_j+1)^2-1}{12}.
\end{align*}
This leads to the approximation
\begin{align*}
  \binom{k}{k\boldsymbol{\mu}}_{g_S}\sim
  \frac{ \Bigl(\prod_j\nu_j+1\Bigr)^{k}}{(2\pi)^{N/2}k\sqrt{|\mathbf{\Sigma|}}}. %=
\end{align*}  
When $N=1$ and $\nu_1=1$ we obtain the well-known approximation $\frac{2^{k+1}}{\sqrt{2\pi k}}$ for the central binomial coefficient $\binom{k}{k/2}$. 
\end{example}

\begin{example}
  Let $S=\set{(0,1),(1,0),(1,1)}$. 
  Let $X_i=(x,y)$ be uniformly distributed on $S$. We have 
  \begin{align*}
    P[x=0]=\frac{1}{3},\:\:
    P[x=1]=\frac{2}{3},\:\: P[xy=0]=\frac{2}{3},\:\:P[xy=1]=\frac{1}{3}. 
  \end{align*}
  Therefore
    $\Cov(x,y)=E[xy]-E[x]E[y]=1/3-(2/3)^2=3/9-4/9=-\frac{1}{9}$. Moreover $\Var(x)=\frac{2}{9}$ and thus %Here $X_1=[x,y]$] 
  \begin{align*}
    \mathbf{\Sigma}=\begin{pmatrix}
    \frac{2}{9} & -\frac{1}{9} \\ -\frac{1}{9} & \frac{2}{9}
    \end{pmatrix},\quad
    \boldsymbol{\mu}=\begin{pmatrix}\frac{2}{3}\\ \frac{2}{3}\end{pmatrix}.
  \end{align*}
  Hence:
  \begin{align*}
     \binom{k}{k\boldsymbol{\mu}}_{g_S}\sim \frac{3^k}{2\pi k\sqrt{\frac{1}{27}}} =
     \frac{3^{k+1}}{2\pi k\sqrt{\frac{1}{3}}}
  \end{align*}
  For example, we have 
  \begin{align*}
    \binom{15}{(10,10)}_{g_S} = 756,756,
  \end{align*}
  while the approximation formula yields $791,096.70\ldots$, which amounts to a
  relative error of less than 5\%. Analogously,
  \begin{align*}
    \binom{18}{(12,12)}_{g_S} = 17,153,136,
  \end{align*}
  while the approximation formula yields $17,799,675.85\ldots$, which amounts to a
  relative error of less than 4\%.
\end{example}
The idea of deriving asymptotics of coefficients via the CLT, that
underlies our above approximations, has been developed in 
different works such as \cite{Eger:2014_central,Walsh:1995}; see
\cite{Mukho:2016} for a survey. While such results can also be
obtained via singularity or saddle point analysis methods using the generating
function for $\binom{k}{\Bell}_f$ in our case
\cite{Flajolet:2009,Ratsaby:2008}, using the CLT with suitably defined
random variables is an alternative that may guarantee additional
desirable properties such as uniform convergence \cite{Neuschel:2014}. 

\section{Basic identities}\label{sec:prelim} 
In the sequel, we write $\mathbf{x}^{\mathbf{s}}$ for
$x_1^{s_1}\cdots x_N^{s_N}$ where $\mathbf{x}=(x_1,\ldots,x_N)$ and
$\mathbf{s}=(s_1,\ldots,s_N)$. 

For $k\ge 0$ and $\ell_1,\ldots,\ell_N\ge 0$ and $f:\nn^N\rightarrow \real$, consider the coefficient of $\mathbf{x}^{\Bell}=x_1^{\ell_1}\cdots x_N^{\ell_N}$ of the  
power series $F$ in the variables $x_1,\ldots,x_N$, where: 
\begin{align}\label{eq:power}
  F(\mathbf{x};k)=\Bigl(\sum_{\mathbf{s}\in\nn^N}
  f(\mathbf{s})\mathbf{x}^{\mathbf{s}}\Bigr)^k, 
\end{align}
and denote it by $[\mathbf{x}^{\Bell}]F(\mathbf{x};k)$. %$\binom{k}{\Bell}_{f}$. 
Our first theorem states that %$\binom{k}{\Bell}_{f}$ 
$[\mathbf{x}^{\Bell}]F(\mathbf{x};k)$ 
denotes the combinatorial object we are investigating in this work,
the number of $f$-weighted vector compositions of $\Bell$ (with a fixed number, $k$, of parts). Therefore $F(\mathbf{x};k)$ is the
generating function for  $\binom{k}{\Bell}_f$. 

\begin{theorem}\label{theorem:main}
  We have that $[\mathbf{x}^{\Bell}]F(\mathbf{x};k) = \binom{k}{\Bell}_f$. 
\end{theorem}
\begin{proof}
  Collecting terms in \eqref{eq:power}, we see that $[\mathbf{x}^{\Bell}]F(\mathbf{x};k)$ is given as 
  \begin{align}\label{eq:comp}
    \sum_{\mathbf{m}_1+\cdots+\mathbf{m}_k=\Bell}f(\mathbf{m}_1)\cdots f(\mathbf{m}_k),
  \end{align}
  where the sum is over all nonnegative vector solutions to
  $\mathbf{m}_1+\cdots+\mathbf{m}_k=\Bell$. Using \eqref{eq:repr1} proves the theorem. 
\end{proof}
Next, we list four identities for $\binom{k}{\Bell}_f$ %the $f$-weighted vector compositions,
which we will make use of in the proofs of (divisibility) properties of the number of vector compositions later
on.  
\begin{theorem}\label{prop:convolution} 
  Let $k\ge 0$ and $\Bell\in\nn^N$. Then, the following hold:
  \begin{align}
    \label{eq:repr}
    \binom{k}{\Bell}_f\: &\:= \sum_{(r_1,r_2,\ldots)\in \mathcal{P}(\Bell;k)}\binom{k}{r_1,r_2,\ldots}\prod_{\mathbf{s}_i}f(\mathbf{s}_i)^{r_i}\\ 
    \label{eq:vandermonde}
    \binom{k}{\Bell}_{f}\: &\:=
    \sum_{\mathbf{q}_1+\cdots+\mathbf{q}_r=\Bell}\binom{k_1}{\mathbf{q}_1}_{f}\cdots\binom{k_r}{\mathbf{q}_r}_{f}
\\
\label{eq:absorption}
\Bell\binom{k}{\Bell}_{f}\: &\:=
\frac{k}{i}\sum_{\mathbf{s}\in\nn^N}\mathbf{s}\binom{i}{\mathbf{s}}_{f}\binom{k-i}{\Bell-\mathbf{s}}_{f}\\
\label{eq:rec}
\binom{k}{\Bell}_{f}\: &\:=
    \sum_{i\in\nn}f(\mathbf{m})^i\binom{k}{i}\binom{k-i}{\Bell-\mathbf{m} i}_{f_{|f(\mathbf{m})=0}}
  \end{align}
  In \eqref{eq:repr}, 
  $\binom{k}{r_1,r_2,\ldots}=\frac{k!}{r_1!r_2!\cdots}$ denote the
  multinomial coefficients. In \eqref{eq:vandermonde}, which we will %is also
   call \emph{Vandermonde convolution}, 
  the sum is over
  all solutions %in nonnegative integers 
  $\mathbf{q}_1,\ldots,\mathbf{q}_r$, $\mathbf{q}_i\in\nn^N$, of
  $\mathbf{q}_1+\cdots+\mathbf{q}_r=\Bell$, and the relationship holds for any fixed
  composition $(k_1,\ldots,k_r)$ of $k$, for $r\ge 1$. In
  \eqref{eq:absorption}, $i$ is an integer satisfying $0<i\le k$. In
  %identity 
  \eqref{eq:rec}, $\mathbf{m}\in\nn^N$ and by $f_{|{f(\mathbf{m})=0}}$ we
  denote the function $g:\nn^N\goesto\nn$ for which $g(\mathbf{s})=f(\mathbf{s})$, for
  all $\mathbf{s}\neq \mathbf{m}$, and $g(\mathbf{m})=0$. 
\end{theorem}
\begin{proof}
  \eqref{eq:repr} follows from rewriting the sum in
  \eqref{eq:comp} as a summation over vector partitions rather than
  over vector compositions and then adjusting the factors in the sum.  
  \eqref{eq:vandermonde} follows because each vector composition of $\Bell$ with $k$
  parts can be subdivided into a fixed number $r$ of `subcompositions' with 
  $k_1,\ldots,k_r$ parts. These %subparts 
  represent weighted vector
  compositions of vectors $\mathbf{q}_i$ with $k_i$ parts and the
  subcompositions are independent of each other, given that the $\mathbf{q}_i$'s
  sum to $\Bell$. 
  In view of our previous discussions, we prove
  \eqref{eq:absorption} for sums of random vectors. For $0<i\le k$, let
  $T_i$ denote the partial sum $X_1+\cdots+X_i$ of i.i.d.\ random
  vectors $X_1,\ldots,X_i,\ldots,X_k$. Consider the
  conditional expectation $\Exp[T_i\sd T_k=n]$, for which the relation
  \begin{align*}
    \Exp[T_i\sd T_k=\Bell] = \frac{\Bell}{k}i,
  \end{align*}
  holds, by independent and identical distribution of
  $X_1,\ldots,X_k$. Moreover, by definition of conditional
  expectation, we have that 
  \begin{align*}
  \Exp[T_i\sd T_k=\Bell] = \sum_{\mathbf{s}\in\nn^N} \mathbf{s} \frac{P[T_i=\mathbf{s},T_k=\Bell]}{P[T_k=\Bell]} = \sum_{\mathbf{s}\in\nn^N} \mathbf{s}\frac{P[T_i=\mathbf{s}]\cdot P[T_{k-i}=\Bell-\mathbf{s}]}{P[T_k=\Bell]}.
\end{align*}
Combining the two identities for $\Exp[T_i\sd T_k=n]$ and rearranging
yields \eqref{eq:absorption}.
To prove \eqref{eq:rec}, let $\mathbf{m}\in\nn^N$. The part value
$\mathbf{m}$ may occur $i=0,\ldots,k$ times in a vector composition of
$\Bell$ with $k$ parts. When it occurs exactly $i$ times we are left with a
composition of $\Bell-\mathbf{m}i$ into $k-i$ parts in which
$\mathbf{m}$ does not occur anymore. The factor $\binom{k}{i}$
distributes the $i$ parts with value $\mathbf{m}$ among $k$ parts and the $i$ parts may be
colored independently into $f(\mathbf{m})$ colors.  
\end{proof}
\begin{remark}\label{rem:triangle}
Note the following important special case of %identity
\eqref{eq:vandermonde} which results when we let $r=2$ and $k_1=1$ and
$k_2=k-1$,
\begin{align*}
  \binom{k}{\Bell}_{f}=\sum_{\mathbf{s}\in\nn^N}f(\mathbf{s})\binom{k-1}{\Bell-\mathbf{s}}_f,
\end{align*}
which establishes that the quantities %weighted integer compositions satisfy a Pascal
$\binom{k}{\Bell}_{f}$ 
may be perceived of as generating a ``Pascal triangle''-like array in
which entries 
in row $k$ are weighted sums of the entries in row
$k-1$. However, note that the entries $\Bell$ in rows $k$ themselves lie in an
$N$-dimensional space.
\end{remark}
We also note the following special cases of $\binom{k}{\Bell}_f$.
\begin{lemma}\label{lemma:f0}
  For all $k\in\nn,\mathbf{x}\in\nn^N$, we have that: 
  \begin{align*}
    \binom{k}{\mathbf{0}}_f &= f(\mathbf{0})^k,\\
    \binom{1}{\mathbf{x}}_f &= f(\mathbf{x}),\\
    \binom{0}{\mathbf{x}}_f &= \begin{cases}1, &
      \text{if } \mathbf{x}=\mathbf{0};\\ 0, & \text{otherwise}.\end{cases}
  \end{align*}
  \qed
\end{lemma}

\newpage 
\section{Combinatorics of partial derivatives}\label{sec:partial}
The formula of Fa\`{a} di Bruno (1825-1888) describes the higher-order
derivatives 
of a composite function $G\circ F$ as a combinatorial sum of the derivatives of
the individual functions $G$ and $F$. Hardy \cite{Hardy:2006} generalizes this
formula to partial derivatives, arguing that treating variables in the
derivatives as distinct is more natural. We provide an alternative
derivation of the partial derivative formula which is based on
interpreting $G\circ F$ as the generating function for weighted
vector compositions. 
As a consequence, %result, 
the formulas for partial derivatives of composite functions 
follow effortlessly from different identities for weighted vector compositions. 

For two power series $G:\real\rightarrow\real$ and $F:\real^N\rightarrow\real$ with $G(z)=\sum_{n\ge 0}g_nz^n$ and $F(\mathbf{z})=\sum_{\mathbf{s}\in\nn^N}f_{\mathbf{s}}\mathbf{z}^{\mathbf{s}}$, we first ask for the power series representation of $G\circ{F}$. We find that
\begin{align*}
  [\mathbf{z}^{\mathbf{s}}](G\circ F)(\mathbf{z}) &= [\mathbf{z}^{\mathbf{s}}]\sum_{n\ge 0}g_n\bigl(\sum_{\mathbf{s}\in\nn^N}f_{\mathbf{s}}\mathbf{z}^{\mathbf{s}}\bigr)^n =  
  \sum_{n\ge 0}g_n[\mathbf{z}^{\mathbf{s}}]\bigl(\sum_{\mathbf{s}\in\nn^N}f_{\mathbf{s}}\mathbf{z}^{\mathbf{s}}\bigr)^n \\
  &= \sum_{n\ge 0}g_n\binom{n}{\mathbf{s}}_f
\end{align*}
by Theorem \ref{theorem:main}. 
Hence, using \eqref{eq:repr1}, we obtain  
\begin{equation}\label{eq:compositionrepr}
  \begin{split}
  (G\circ F)(\mathbf{z}) &= \sum_{\mathbf{s}\in\nn^N}\mathbf{z}^{\mathbf{s}}\left(\sum_{n\ge 0}g_n\binom{n}{\mathbf{s}}_f\right) \\
  &=  \sum_{\mathbf{s}\in\nn^N}\mathbf{z}^{\mathbf{s}}\left(\sum_{n\ge 0}g_n\sum_{\pi\in \mathcal{C}(\mathbf{s};n)}f_{\mathbf{m}_1}\cdots f_{\mathbf{m}_n}\right)\\
  &=  \sum_{\mathbf{s}\in\nn^N}\mathbf{z}^{\mathbf{s}}\left(\sum_{n\ge 0}\sum_{\pi\in \mathcal{C}(\mathbf{s};n)}g_nf_{\mathbf{m}_1}\cdots f_{\mathbf{m}_n}\right) \\ &= \sum_{\mathbf{s}\in\nn^N}\mathbf{z}^\mathbf{s}\left(\sum_{\pi\in\mathcal{C}(\mathbf{s})}g_{|\pi|}f_\pi\right)
  \end{split}
\end{equation}
where we let $\mathcal{C}(\mathbf{s};n)$ stand for $\set{\pi=(\mathbf{m}_1,\ldots,\mathbf{m}_n)\,|\, \mathbf{m}_i\in\nn^N, \sum_{i=1}^n \mathbf{m}_i=\mathbf{s}}$ (vector compositions of $\mathbf{s}$ with fixed number $n$ of parts) and $\mathcal{C}(\mathbf{s})$ analogously represents the class of vector compositions of $\mathbf{s}$ with arbitrary number of parts. Moreover, we use the abbreviation $f_{\pi}=f_{\mathbf{m}_1}\cdots f_{\mathbf{m}_n}$ and $|\pi|$ denotes the number of parts in $\pi$. Note that the above representation generalizes the analogous representation derived in Vignat and Wakhare \cite{Vignat:2017} to the multivariate case. 

Since
\begin{align*}
  \frac{1}{\Bell!}\frac{\partial^{||\Bell||} H(\mathbf{0})}{\partial\mathbf{z}^{\Bell}} = [\mathbf{z}^{\Bell}]H(\mathbf{z}), 
\end{align*}
for any power series $H(\mathbf{z})$, 
we immediately have several Fa\`{a} di Bruno like representations of partial derivatives. Here, we write $\partial \mathbf{z}^{\Bell}$ for $\partial z_1^{\ell_1}\cdots\partial z_N^{\ell_N}$, $||\Bell||$ for $\ell_1+\cdots+\ell_N$ and $\Bell!$ for $\ell_1!\cdots\ell_N!$. 
\begin{theorem}\label{theorem:deriv_alt}
Let $G\circ F:\real^N\rightarrow \real$, with $F:\real^N\rightarrow \real$ and
$G:\real\rightarrow\real$. Let $\Bell=(\ell_1,\ldots,\ell_N)\in\nn^N$ and 
assume that $G$ and $F$ have a sufficient number of derivatives.  Then 
\begin{align*}
  \frac{\partial^{||\Bell||}(G\circ F)(\mathbf{x})}{\partial \mathbf{z}^{\Bell}}
  &=
  \sum_{\pi=(\mathbf{m}_1,\mathbf{m}_2,\ldots)\in \mathcal{C}(\Bell)}\frac{\Bell!}{|\pi|!\mathbf{m}_1!\mathbf{m}_2!\cdots}G^{(|\pi|)}(F(\mathbf{x}))\prod_i\frac{\partial^{||\mathbf{m}_i||} F(\mathbf{x})}{\partial \mathbf{z}^{\mathbf{m}_i}}
%\frac{\partial^{||\mathbf{m}_2||} F(\mathbf{x})}{\partial \mathbf{z}^{\mathbf{m}_2}}\cdots
  \end{align*}
  \qed
\end{theorem}
Note that in the theorem, terms $\frac{\partial^{||\mathbf{m}_i||} F(\mathbf{x})}{\partial \mathbf{z}^{\mathbf{m}_i}}$ with $\mathbf{m}_i=\mathbf{0}$ drop, so we can perceive of the sum as being over non-zero parts $\mathbf{m}_i$. 

Alternative representations of the partial derivative can be derived by considering different identities for $\binom{n}{\mathbf{s}}_f$. For example, using 
\eqref{eq:repr}, we obtain Theorem \ref{theorem:deriv} below. Still other representations follow analogously from considering further identities of $\binom{n}{\mathbf{s}}_f$, e.g., \eqref{eq:rec}, plugged into the representation of $(G\circ F)(\mathbf{z})$ in \eqref{eq:compositionrepr} above.  
\begin{theorem}\label{theorem:deriv}
Let $G\circ F:\real^N\rightarrow \real$, with $F:\real^N\rightarrow \real$ and
$G:\real\rightarrow\real$. Let $\Bell=(\ell_1,\ldots,\ell_N)\in\nn^N$ and 
assume that $G$ and $F$ have a sufficient number of derivatives.  Then 
\begin{align*}
   \frac{\partial^{||\Bell||}(G\circ F)(\mathbf{x})}{\partial \mathbf{z}^{\Bell}}
  &=
  \sum_{(r_1,r_2,\ldots)\in \mathcal{P}^{(\mathbf{S}(\Bell))}({\Bell})}
  \frac{\Bell!}{r_1!r_2!\cdots}G^{(r)}(F(\mathbf{x}))\prod_{i}  
  \left(\frac{1}{\mathbf{s}_i!}\frac{\partial^{||\mathbf{s}_i||}
    F(\mathbf{x})}{\partial\mathbf{z}^{\mathbf{s}_i}}\right)^{r_i}
\end{align*}
where $r=r_1+r_2+\cdots$.\qed
\end{theorem}
\begin{example}
  Let $\Bell=(1,2)$. Then
  $S(\Bell)=\set{(0,1),(1,0),(1,1),(1,2),(0,2)}$. Moreover,
  \begin{align*}
  \mathcal{P}^{(S(\Bell))}(\Bell)=\set{(0,0,0,1,0),(1,0,1,0,0),(0,1,0,0,1),(2,1,0,0,0)}. 
  \end{align*}
  Therefore, according to Theorem \ref{theorem:deriv}
  \begin{align*}
    \frac{\partial^3(G\circ F)}{\partial x\partial y^2} &= 2\left(
    G'(F(\mathbf{x}))\frac{1}{2}\frac{\partial^3 F(\mathbf{x})}{\partial x\partial y^2}+G''(F(\mathbf{x}))\frac{\partial F(\mathbf{x})}{\partial y}\frac{\partial^2 F(\mathbf{x})}{\partial x\partial y}\right.\\
    &+\left. G''(F(\mathbf{x}))\frac{\partial F(\mathbf{x})}{\partial x}\frac{1}{2}\frac{\partial^2 F(\mathbf{x})}{\partial y^2}+
    \frac{1}{2}G'''(F(\mathbf{x}))\Bigl(\frac{\partial F(\mathbf{x})}{\partial y}\Bigr)^2\frac{\partial F(\mathbf{x})}{\partial x}
    \right). 
  \end{align*}
\end{example}

We now show that Theorem \ref{theorem:deriv_alt} (or equivalently Theorem \ref{theorem:deriv}) generalizes the main formula derived in \cite{Hardy:2006}. Recall that a \emph{set partition} of $[n]=\set{1,\ldots,n}$ is a set of disjoint, non-empty subsets of $[n]$ whose union is $[n]$.

\begin{lemma}\label{lemma:setp}
  There is a bijection between the set of all vector partitions (unordered compositions) of the vector $\underbrace{(1,\ldots,1)}_{{n \text{ times}}}$ into $k$ non-zero parts and the set of all {set partitions} of $[n]$ into $k$ parts. 
\end{lemma}
The proof of the lemma is straightforward. We can assign each set partition $a=\set{a_1,\ldots,a_k}$ (where $a_i\subseteq[n]$, $a_i\neq\emptyset$, $\bigcup_i a_i=[n], a_i\cap a_j=\emptyset$) the vector partition $\mathbf{b}_1+\cdots+\mathbf{b}_k$ where $\mathbf{b}_i$ is a vector whose entries are $1$ for all indices in $a_i$ and zero otherwise (and vice versa). Due to the properties of $a$, $\mathbf{b}_1+\cdots+\mathbf{b}_k$ yields $(1,\ldots,1)$. 

Further, since the parts of each vector partition of $\mathbf{1}=(1,\ldots,1)$ into $k$ non-zero parts are all distinct, we also have that $|\mathcal{C}(\mathbf{1};k)|=k!|\mathcal{P}(\mathbf{1};k)|$.

To derive the main result in \cite{Hardy:2006}, we now let $\Bell$ in Theorem \ref{theorem:deriv_alt} be $\mathbf{1}=(1,\ldots,1)$ (each of $N$ variables occurs exactly once). Then $||\Bell||=N$ and $\Bell!=1$ and $\mathbf{m}_i!=1$. Thus,\footnote{In the equation, we perceive of $\mathcal{P}(\Bell)$ as directly containing unordered vectors, rather than multiplicities as in our original definition.} 
\begin{align*}
\frac{\partial^{N}(G\circ F)(\mathbf{x})}{\partial \mathbf{z}^{\mathbf{1}}}
&= 
\sum_{\pi=(\mathbf{m}_1,\mathbf{m}_2,\ldots)\in \mathcal{C}(\mathbf{1})}\frac{1}{|\pi|!}G^{(|\pi|)}(F(\mathbf{x}))\frac{\partial^{||\mathbf{m}_1||} F(\mathbf{x})}{\partial \mathbf{z}^{\mathbf{m}_1}}\frac{\partial^{||\mathbf{m}_2||} F(\mathbf{x})}{\partial \mathbf{z}^{\mathbf{m}_2}}\cdots\\
&=\sum_{\pi=(\mathbf{m}_1,\mathbf{m}_2,\ldots)\in \mathcal{P}(\mathbf{1})}G^{(|\pi|)}(F(\mathbf{x}))\frac{\partial^{||\mathbf{m}_1||} F(\mathbf{x})}{\partial \mathbf{z}^{\mathbf{m}_1}}\frac{\partial^{||\mathbf{m}_2||} F(\mathbf{x})}{\partial \mathbf{z}^{\mathbf{m}_2}}\cdots
\end{align*}
Interpreting the last quantity as a sum over set partitions, using Lemma \ref{lemma:setp}, 
with 
the $\mathbf{m}_i$ as subsets of $[N]$ yields the formula (5) in \cite{Hardy:2006}. 

Correspondingly, our representation in Theorem \ref{theorem:deriv} is the direct analogue of the representation in \cite{Hardy:2006} based on `multiset partitions' (Corollary to Propositions 1 and 2 in \cite{Hardy:2006} combined with Proposition 4 therein). 

There has been some debate on the combinatorial nature of higher-order derivatives. While they
may (thus) be perceived of as set partitions \cite{Hardy:2006,Johnson:2002}, Yang
\cite{Yang:2000} finds that they are ``essentially integer
partitions''. 
Noting the relationships and equivalences between these concepts and based on our derivations, 
we may also claim that partial
derivatives of composite functions are essentially vector compositions! 

\section{Congruences for $\binom{k}{\Bell}_f$ }\label{sec:congruence1}
\begin{theorem}[Parity of $\binom{k}{\Bell}_f$]\label{theorem:parity}
  Let $k\ge 0$ and let $\Bell\in\nn^N$. Then
  \begin{align*}
    \binom{k}{\Bell}_{f} \equiv
    \begin{cases}
      0 \pmod{2}, & \text{if $k$ is even and}\\
                  & \text{$\Bell$ has at least one}\\
                  & \text{odd entry};\\
      \binom{k/2}{\Bell/2}_{f} \pmod{2}, & \text{if $k$ is even and}\\
        & \text{$\Bell$ has only even entries};\\
      \sum_{\set{\mathbf{s}\sd \Bell-\mathbf{s} \text{ has only even entries}}}f(\mathbf{s})\binom{\lfloor k/2\rfloor}{\frac{\Bell-\mathbf{s}}{2}}& \text{if $k$ is odd}.
    \end{cases}
  \end{align*}
\end{theorem}
\begin{proof} 
  We distinguish three cases. 
  \begin{itemize}
    \item Case $1$: 
	Let $k$ be even and let one entry of $\Bell$ be odd. Consider 
        \eqref{eq:absorption} in Theorem \ref{theorem:main}
        with $i=1$. 
        If $k$ is even, the right-hand
      side vector is even in each entry. 
      Thus, if $\Bell$ is odd in one entry, $\binom{k}{\Bell}_{f}$ must
      be even.  
    \item Case $2$: Let $k$ be even and $\Bell$ be even in each entry. Consider the
      Vandermonde convolution in the case of $r=2$ and $k_1=k_2=k/2$. Then,
      \begin{align*}
        \binom{k}{\Bell}_{f} &=
        \sum_{\mathbf{a}+\mathbf{b}=\Bell}\binom{k/2}{\mathbf{a}}_{f}\binom{k/2}{\mathbf{b}}_{f}.  
      \end{align*}
      All pairs $(\mathbf{a},\mathbf{b})$ for which
      $\mathbf{a}\neq\mathbf{b}$ occur exactly twice, so their sum contributes 
      nothing modulo $2$. The only term that does not occur twice is
      $\mathbf{a}=\mathbf{b}$, for which $\mathbf{a}=\Bell/2$. Hence,
      \begin{align*}
        \binom{k}{\Bell}_{f} &\equiv \binom{k/2}{\Bell/2}_f^2 \equiv
        \binom{k/2}{\Bell/2}_f\pmod{2}. 
      \end{align*}
    \item Case $3$: Let $k$ be odd. Then $k-1$ is even. Thus,
      %consider 
      the Vandermonde convolution with $k_1=1$, $r=2$ implies 
      \begin{align*}
        \binom{k}{\Bell}_{f}&=\sum_{\mathbf{s}\in
          \nn^N}f(\mathbf{s})\binom{k-1}{\Bell-\mathbf{s}}_{f}\\
        &\equiv 
         \sum_{\set{\mathbf{s}\sd \Bell-\mathbf{s} \text{ has only even entries}}}f(\mathbf{s})\binom{\lfloor k/2\rfloor}{\frac{\Bell-\mathbf{s}}{2}}
        \pmod{2},
      \end{align*}
      where we use Case $1$ and Case $2$ in the last
      congruence. 
  \end{itemize}
\end{proof}
\begin{example}
  Let $f((0,1,0))=3$ and let 
  $f(\mathbf{s})=1$ for all \\
  $\mathbf{s}\in\set{(1,0,0),(0,0,1),(1,1,0),(1,0,1), (0,1,1),(1,1,1)}$. Let $f(\mathbf{s})=0$ for all other $\mathbf{s}$. Then, by  
  Theorem \ref{theorem:parity},
  \begin{align*}
    \binom{21}{(20,19,18)}_f &\equiv f((0,1,0))\binom{10}{(10,9,9)}_f\equiv 0\pmod{2}. 
  \end{align*}
  In fact, $\binom{21}{(20,19,18)}_f=7,301,700$. In contrast,
  \begin{align*}
    \binom{19}{(3,16,2)}_f &\equiv \binom{9}{(1,8,1)}_f\equiv \binom{4}{(0,4,0)}_f \equiv \binom{2}{(0,2,0)}_f \\ &\equiv \binom{1}{(0,1,0)}_f \equiv 1\pmod{2}.
  \end{align*}
  Indeed, $\binom{19}{(3,16,2)}_f=8,356,358,620,683$.
\end{example}

\begin{theorem}\label{theorem:prime1}
  Let $p$ be prime, $\Bell\in\nn^N$. Then %and let $m\ge 1$. Then
  \begin{align*}
    \binom{p}{\Bell}_{f} \equiv 
    \begin{cases}
      f(\mathbf{m}) \pmod{p}, & \text{if $\Bell=\mathbf{m}p$ for some $\mathbf{m}$};\\ %$0\le r\le n$},\\
      0 \pmod{p}, & \text{else.}
    \end{cases}
  \end{align*}
\end{theorem}
We sketch three proofs of Theorem \ref{theorem:prime1}, 
a combinatorial proof and two proof sketches based on identities in
Theorem \ref{prop:convolution}. The first proof uses the
following lemma (see \cite{Anderson:2005}). 
\begin{lemma}\label{lemma:comb}
  Let $S$ be a finite set, let $p$ be prime, and suppose $g:S\goesto
  S$ has the property that $g^p(x)=x$ for any $x$ in $S$, where $g^p$
  is the $p$-fold composition of $g$. Then
  $|{S}|\equiv|{F}|\pmod{p}$, where $F$ is the set of fixed points
  of $g$. \qed
\end{lemma}
\begin{proof}[Proof of Theorem \ref{theorem:prime1}, 1]
  Let $g$, a map from the set of $f$-weighted vector compositions of $\Bell$ with $p$ parts to itself, 
  be the operation that shifts all parts one to the right, modulo
  $p$. In other words, $g$ maps (denoting colors by superscripts)
  $[\mathbf{m}_1^{\alpha_1},\mathbf{m}_2^{\alpha_2},\ldots,\mathbf{m}_{p-1}^{\alpha_{p-1}},\mathbf{m}_p^{\alpha_p}]$
  to 
  \begin{align*}
    [\mathbf{m}_p^{\alpha_p},\mathbf{m}_1^{\alpha_1},\mathbf{m}_2^{\alpha_2},\ldots,\mathbf{m}_{p-1}^{\alpha_{p-1}}].
  \end{align*}
  Of course, applying $g$ $p$ times yields the original %$f$-colored
  vector composition, that is, $g^{p}(x)=x$ for all $x$. %For $k=p$
  %prime, 
	We may thus apply Lemma \ref{lemma:comb}. If $\Bell$ allows a
  representation $\Bell=p\mathbf{m}$ for some suitable $\mathbf{m}$, $g$ has exactly $f(\mathbf{m})$
  fixed points, namely, all compositions $\underbrace{[\mathbf{m}^{1},\ldots,
    \mathbf{m}^{1}]}_{p \text{ times}}$ to
  $\underbrace{[\mathbf{m}^{f(\mathbf{m})},\ldots,\mathbf{m}^{f(\mathbf{m})}]}_{p \text{ times}}$. Otherwise,
  if $\Bell$ has no such 
  representation, $g$ has no fixed points. 
  This proves the theorem. 
\end{proof}
\begin{proof}[Proof of Theorem \ref{theorem:prime1}, 2]
  We apply \eqref{eq:rec} in Theorem 
  \ref{prop:convolution}. Since for the ordinary binomial
  coefficients, the relation $\binom{p}{n}\equiv
  0\pmod{p}$ holds for all $1\le n\le p-1$ and
  $\binom{p}{0}=\binom{p}{p}=1$, we have 
  \begin{align*}
    \binom{p}{\Bell}_{f}&\equiv
    \binom{p}{\Bell}_{f_{|f(\mathbf{m})=0}}+f(\mathbf{m})^p\binom{0}{\Bell-\mathbf{m}
      p}_{f_{|f(\mathbf{m})=0}}\\
      &\equiv \binom{p}{\Bell}_{f_{|f(\mathbf{m})=0}}+f(\mathbf{m})\binom{0}{\Bell-\mathbf{m}
      p}_{f_{|f(\mathbf{m})=0}} \pmod{p}, 
  \end{align*}
  for any $\mathbf{m}$ and where the last congruence is due to Fermat's
  little theorem. Therefore, if $\Bell=\mathbf{m}p$ for some $\mathbf{m}$, then
  $\binom{p}{\Bell}_f\equiv  \binom{p}{\Bell}_{f_{|f(\mathbf{m})=0}}+f(\mathbf{m})\pmod{p}$ and
  otherwise $\binom{p}{\Bell}_f \equiv
  \binom{p}{\Bell}_{f_{|f(\mathbf{m})=0}}\pmod{p}$ for any $\mathbf{m}$. 
  Now, the theorem follows inductively. 
\end{proof}
\begin{proof}[Proof of Theorem \ref{theorem:prime1}, 3] 
  We use 
  \eqref{eq:repr} in Theorem \ref{prop:convolution} in conjunction
  with the following property of multinomial
  coefficients (see, e.g., \cite{Ricci:1931}): 
  \begin{align}\label{eq:multi}
    \binom{k}{k_1,k_2,\ldots}\equiv
    0\pmod{\frac{k}{\gcd{(k_1,k_2,\ldots)}}}. 
  \end{align}
  Since the multiplicities $r_1,r_2,\ldots$ for $\binom{p}{\Bell}_f$ in \eqref{eq:repr} satisfy
  $r_1+r_2+\cdots = p$, we have $d=\gcd{(r_1,r_2,\ldots)}\in\set{1,p}$, since otherwise $p$ was composite. Moreover, $d=p$ % then 
  if and only if exactly one of the $r_i$ equals $p$ and all the other are zero. 
  Hence, whenever $\Bell\neq p\mathbf{m}$, for any $\mathbf{m}$, then $d=1$ for all $(r_1,r_2,\ldots)$ in the summation, for otherwise, the condition $r_1\mathbf{s}_1+r_2\mathbf{s}_2+\cdots=\Bell$ would imply that $p\mathbf{s}_i=\Bell$, a contradiction. Therefore, $\binom{p}{\Bell}_f\equiv 0\pmod{p}$
  since all terms in the summation in \eqref{eq:repr} are congruent to zero modulo $p$ by \eqref{eq:multi}.
  Consider now the case $\Bell=p\mathbf{m}$ for some $\mathbf{m}$.
  Then, $\mathbf{m}\in S(\Bell)$, that is, $\mathbf{m}=\mathbf{s}_i$ for some $i$. Again, the only terms in the summation that contribute modulo $p$ are those for which $d=p$. Thus, there is exactly one term that contributes, namely, $(r_1,r_2,\ldots,r_i,\ldots)=(0,0,\ldots,p,\ldots)$. %and $\mathbf{s}_i=\mathbf{m}$.  
  Therefore, $\binom{p}{\Bell}_f\equiv \binom{p}{0,\ldots,0,p,0,\ldots}f(\mathbf{\mathbf{m}})^p \equiv f(\mathbf{m})\pmod{p}$.
\end{proof}
We call the next congruence Babbage's congruence, since Charles
Babbage was apparently the first to assert the respective congruence
in the case of ordinary binomial coefficients \cite{Babbage:1819}. 
\begin{theorem}[Babbage's congruence]\label{theorem:babbage}
 Let $p$ be prime, let $n$ be a nonnegative integer, and let $\mathbf{m}\in\nn^N$. Then
\begin{align*}
  \binom{np}{\mathbf{m}p}_{f} \equiv \binom{n}{\mathbf{m}}_{g}\pmod{p^2},
\end{align*}  
whereby $g$ is defined as $g(\mathbf{x})=\binom{p}{\mathbf{x}p}_{f}$, for all $\mathbf{x}$. 
\end{theorem}
\begin{proof}
By the Vandermonde convolution, we have
\begin{align}\label{eq:p2}
  \binom{np}{\mathbf{m}p}_f = \sum_{\mathbf{k}_1+\cdots+\mathbf{k}_n=\mathbf{m}p}\binom{p}{\mathbf{k}_1}_f\cdots\binom{p}{\mathbf{k}_n}_f
\end{align}
Now, by Theorem \ref{theorem:prime1}, $p$ divides
$\binom{p}{\mathbf{x}}_{f}$ %unless 
whenever $\mathbf{x}$ is not of the form $\mathbf{x}=\mathbf{r}p$. Hence, modulo
$p^2$, the only terms that contribute to the sum are those for which
at least $n-1$ $\mathbf{k}_i$'s are of the form
$\mathbf{k}_i=\mathbf{r}_ip$. Since the  
$\mathbf{k}_i$'s must sum to $\mathbf{m}p$, this implies that all $\mathbf{k}_i$'s are of the
form $\mathbf{k}_i=\mathbf{r}_ip$, for $i=1,\ldots,n$. Hence, 
modulo $p^2$, \eqref{eq:p2} becomes 
\begin{align*}
   \sum_{{\mathbf{r}_1+\cdots+\mathbf{r}_n=\mathbf{m}}}\prod_{i=1}^n \binom{p}{\mathbf{r}_ip}_{f} =\sum_{{\mathbf{r}_1+\cdots+\mathbf{r}_n=\mathbf{m}}}\prod_{i=1}^n g(\mathbf{r}_i),
\end{align*}
The last sum is precisely $\binom{n}{\mathbf{m}}_{g}$.
\end{proof}
\begin{example}\label{example:170}
  Let $f$ be the indicator function on the set \\
  $\set{(1,0),(0,1),(1,1),(2,1),(1,2)}$. Let $p=3$, $n=2$, and $\mathbf{m}=(1,2)$. Enumeration shows that 
  \begin{align*}
    \binom{6}{(3,6)}_f = 170.
  \end{align*}
  Moreover, $\binom{2}{(1,2)}_g$ can be determined by looking at the compositions of $(1,2)$ in two parts, which are $(1,2)=(0,1)+(1,1)=(1,1)+(0,1)$. We have $g((1,1))=\binom{3}{(3,3)}_f=13$ and $g((0,1))=\binom{3}{(0,3)}_f=1$. Hence, $\binom{2}{(1,2)}_g = 26 \equiv 8\equiv \binom{6}{(3,6)}_f\pmod{3^2}$, as predicted. 
\end{example}
Since $g(\mathbf{m})\equiv f(\mathbf{m})\pmod{p}$, by Theorem \ref{theorem:prime1}, we have the following theorem.
\begin{theorem}\label{theorem:special}
Let $p$ be prime, let $n$ be a nonnegative integer, and let $\mathbf{m}\in\nn^N$. Then
\begin{align*}
  \binom{np}{\mathbf{m}p}_{f} \equiv \binom{n}{\mathbf{m}}_{f}\pmod{p}.
\end{align*}  
\qed
\end{theorem}
We use Theorem \ref{theorem:special} to prove a stronger version
of Theorem \ref{theorem:prime1}, namely:
\begin{theorem}\label{theorem:prime2}
  Let $p$ be prime and let $m\ge 1$, $\Bell\in\nn^N$. Then
  \begin{align*}
    \binom{p^m}{\Bell}_{f} \equiv 
    \begin{cases}
      f(\mathbf{m}) \pmod{p}, & \text{if $\Bell=p^m\mathbf{m}$ for some $\mathbf{m}$};\\ %$0\le r\le n$},\\
      0 \pmod{p}, & \text{else.}
    \end{cases}
  \end{align*}
\end{theorem}
\begin{proof}
  Let $\Bell=p^m\mathbf{m}$. Using Theorem 
\ref{theorem:special} twice, %and \ref{theorem:prime1}, 
we find for $m=2$
  \begin{align*}
    \binom{p^2}{p^2\mathbf{m}}_f \equiv \binom{p}{p\mathbf{m}}_f
    \equiv f(\mathbf{m})\pmod{p}. 
  \end{align*}
  Using this, we find that: 
  \begin{align*}
    \binom{p^3}{p^3\mathbf{m}}_f \equiv \binom{p^2}{p^2\mathbf{m}}_f 
    \equiv 
    f(\mathbf{m})\pmod{p},
  \end{align*}
  and so on for any $m$.

  Consider now the case $\Bell\neq p^m\mathbf{m}$ for any $\mathbf{m}$. We use \eqref{eq:rec} from Theorem \ref{prop:convolution} together with the fact that $\binom{p^m}{n}\equiv 0\pmod{p}$ when $0< n<p^m$ and $\equiv 1 \pmod{p}$ whenever $n=1,p^m$. From this it follows that
\begin{align*}
  \binom{p^m}{\Bell}_f \equiv \binom{p^m}{\Bell}_{f|_{f(\mathbf{m})=0}}\pmod{p},
\end{align*}
for any $\mathbf{m}$. We can successively set all arguments of $f$ to zero and note that hence $\binom{p^m}{\Bell}_f \equiv 0\pmod{p}$. 
\end{proof}
Now, we consider the case when $\Bell$ in $\binom{np}{\Bell}_f$ is not of the
form $\mathbf{m}p$ for any $\mathbf{m}$. 
\begin{theorem}\label{theorem:babbage2}
  Let $p$ be prime and let $n$ be a nonnegative integer. Let $\Bell$ not be of the form $\Bell=p\mathbf{m}$, for any $\mathbf{m}$. Then 
  \begin{align*}
    \binom{np}{\Bell}_f \equiv
    n\cdot%\sum_{\substack{i_1=0\\ \divides{p}{r-i_1}}}^r
    \sum_{\set{\mathbf{k}\in S(\Bell)\sd p\nmid\mathbf{k},\Bell-\mathbf{k}=\mathbf{x}p}}
    \binom{p}{\mathbf{k}}_f\binom{n-1}{\mathbf{x}}_g \pmod{p^2},
  \end{align*}
  where $g$ is as defined in Theorem \ref{theorem:babbage}. %above. 
\end{theorem}
\begin{proof}
  By the Vandermonde convolution, \eqref{eq:vandermonde}, we find
  that
  \begin{align*}
    \binom{np}{\Bell}_f &=
    \sum_{\mathbf{k}_1+\cdots+\mathbf{k}_n=\Bell}\binom{p}{\mathbf{k}_1}_f\cdots\binom{p}{\mathbf{k}_n}_f \\
    &= \sum_{\mathbf{k}\in S(\Bell)} \binom{p}{\mathbf{k}}_f
    \sum_{\mathbf{k}_2+\cdots+\mathbf{k}_n=\Bell-\mathbf{k}}\binom{p}{\mathbf{k}_2}_f\cdots\binom{p}{\mathbf{k}_n}_f. 
  \end{align*}
  As in the proof of Theorem \ref{theorem:babbage2}, at least $n-1$ factors $\binom{p}{\mathbf{k}_j}_f$ must be such that $\mathbf{k}_j=\mathbf{r}_jp$.  
  Not all $n$ factors can be of the form
  $\mathbf{r}_jp$, since otherwise $\mathbf{k}_1+\cdots+\mathbf{k}_n=p(\mathbf{r}_1+\cdots+\mathbf{r}_n)=\Bell$, a contradiction. 
  %contradicting that $p\nmid r$. 
  Hence, exactly $n-1$ factors must be
  of the form $\mathbf{r}_jp$, and therefore, 
  \begin{align*}
    \binom{np}{\Bell}_f & \equiv n\sum_{\mathbf{k}\in S(\Bell),\mathbf{k}\neq \mathbf{r}p} \binom{p}{\mathbf{k}}_f
    \sum_{\mathbf{r}_2p+\cdots+\mathbf{r}_np=\Bell-\mathbf{k}}\binom{p}{\mathbf{r}_2p}_f\cdots\binom{p}{\mathbf{r}_np}_f
    \\
    \\
    &= n\sum_{\mathbf{k}\in S(\Bell),p\nmid \mathbf{k}} \binom{p}{\mathbf{k}}_f
    \sum_{\mathbf{r}_2p+\cdots+\mathbf{r}_np=\Bell-\mathbf{k}}g(\mathbf{r}_2)\cdots
    g(\mathbf{r}_n)\pmod{p^2}. 
  \end{align*}
  Now, the
  equation $p(\mathbf{r}_2+\cdots+\mathbf{r}_n)=\Bell-\mathbf{k}$ has solutions
  if and only if $\divides{p}{\Bell-\mathbf{k}}$, that is, when there exists $\mathbf{x}$
  such that $\Bell-\mathbf{k}=\mathbf{x}p$. %, and note that $\binom{1}{x}_h=h(x)$.  
\end{proof}
\begin{example}
  Let $n=4$, $p=3$ and $\Bell=(2,3)$. In this situation, the only
  suitable $\mathbf{k}$ in the previous theorem is
  $\mathbf{k}=(2,3)$ to which corresponds $\mathbf{x}=(0,0)$. The theorem thus implies that
  \begin{align*}
    \binom{12}{(2,3)}_f \equiv
    4\cdot\binom{3}{(2,3)}_f\cdot\binom{3}{(0,0)}_g\pmod{p^2}. 
  \end{align*}
  Let $f(\mathbf{s})=s_1+s_2+1$ for all
  $\mathbf{s}\in\set{(0,0),(0,1),(1,0),(1,1)}$ and $f(\mathbf{s})=0$ otherwise. Then
  $\binom{3}{(0,0)}_g=1$ since
  $g((0,0))=\binom{3}{(0,0)}_f=1$. Moreover,
  $\binom{3}{(2,3)}_f=54$. Therefore 
  \begin{align*}
    4\cdot\binom{3}{(2,3)}_f\cdot\binom{3}{(0,0)}_g \equiv 0\pmod{9}.
  \end{align*}
  Indeed,
  \begin{align*}
    \binom{12}{(2,3)}_f = 407,880 = 45,320\cdot 9.
  \end{align*}
\end{example}
\begin{theorem}\label{theorem:divis}
Let $k\ge 0$, $\Bell\in \mathbb{N}^N$. %Then
Let $d_i=\gcd(k,\ell_i)$ and let $t_i=\frac{k}{d_i}$. Then
\begin{align*}
\binom{k}{\Bell}_{f}\equiv 0\pmod{t_i}
\end{align*}
for all $i=1,\ldots,N$. Equivalently,
\begin{align*}
\binom{k}{\Bell}_{f}\equiv 0\pmod{M}.
\end{align*}
Here, $M$ is the number $M=p_1^{m_1}\cdots p_R^{m_R}$, where the $t_i$
have prime factorization $t_i=p_1^{{(a_i)}_1}\cdots p_R^{{(a_i)}_R}$
and where $m_j=\max_i {(a_i)}_j$, for all $j=1,\ldots,R$. 
\end{theorem}
\begin{proof}
From \eqref{eq:absorption}, with $i=1$, write 
\begin{align*}
  {\Bell}\binom{k}{\Bell}_{f} = {k}\underbrace{\sum_{\mathbf{s}}\mathbf{s}f(\mathbf{s})\binom{k-1}{\Bell-\mathbf{s}}_{f}}_{=:\mathbf{m}\in\nn^N}.
\end{align*}
Now, for any $1\le i\le N$, consider this equation at component $i$, dividing by $d_i=\gcd(k,\ell_i)$:
\begin{align*}
  \frac{\ell_i}{d_i}\binom{k}{\Bell}_f = \frac{k}{d_i}m_i.
\end{align*}
Since $\gcd(k/d_i,\ell_i/d_i)=1$, this means that $\frac{k}{d_i}\mid \binom{k}{\Bell}_f$ for all $i=1,\ldots,N$. %The theorem follows.  
\end{proof}
\begin{example}
  Let $f$ be the indicator function on the set \\
  $\set{(1,0),(0,1),(1,1),(1,2),(2,1),(0,0)}$. Enumeration shows that
  \begin{align*}
    \binom{12}{(9,8)}_f = 44,742,060
  \end{align*}
  We have $t_1=12/3=4$ and $t_2=12/4=3$. Hence $4\cdot 3$ divides $\binom{12}{(9,8)}_f$, and indeed, $44,742,060 = 12\cdot 3,728,505$. 
\end{example}

\begin{theorem}\label{theorem:ms}
  Let $p$ be prime, $n\ge 1$ arbitrary.
  Then, 
\begin{align*}
 \binom{pn}{p\mathbf{\mathbf{1}}}_{f}\equiv \sum_{k=1}^n\sum_{(r_1,r_2,\ldots)\in\mathcal{P}(\mathbf{1};k)}&\frac{(pn)!}{(pr_1)!(pr_2)!\cdots(p(n-k))!}\cdot 
   \\ &f(\mathbf{0})^{p(n-k)}h(\mathbf{s}_1)h(\mathbf{s}_2)\cdots\pmod{pn},
\end{align*}
where $h(\mathbf{s})=\begin{cases}f(\mathbf{s})^p,& \text{if } \mathbf{s}\in U; \\ 0, & \text{else};\end{cases}$ for $U=\set{\mathbf{x}\neq\mathbf{0}\in \nn^N\sd x_i\in\set{0,1}}$.
\end{theorem}
In the theorem, note that $\frac{(pn)!}{(pr_1)!(pr_2)!\cdots(p(n-k))!}=\frac{(pn)!}{(p!)^k(p(n-k))!}$. Also note that the limit of the summation for $k$ is (more adequately described as) $\min\set{n,N}$.  
\begin{proof}
From \eqref{eq:repr},
$\binom{pn}{p\mathbf{1}}_{f}$ can be written as 
\begin{align}\label{eq:help}
  \binom{pn}{p\mathbf{\mathbf{1}}}_{f} =
  \sum_{\substack{r_1+r_2+\cdots=pn,\\ \sum_{\mathbf{s}_i\in S(p\mathbf{1})} r_i\mathbf{s}_i=p\mathbf{1}}}\binom{pn}{r_1,r_2,\ldots}\prod_{\mathbf{s}_i\in S(p\mathbf{1})} f(\mathbf{s}_i)^{r_i}. 
\end{align}
For a term in the sum, %assume that $d=\gcd(k_0,\ldots,k_p)>1$. 
either $d=\gcd(r_1,r_2,\ldots)=1$ or $d=p$, since otherwise, if $1<d<p$, 
then, $d\cdot\sum_{\mathbf{s}_i\in S(p\mathbf{1})} \frac{r_i}{d}\mathbf{s}_i=p\mathbf{1}$, whence $p$ is composite, a contradiction.
Those terms on the RHS of \eqref{eq:help} for which $d=1$ contribute
nothing to the sum modulo $pn$, by \eqref{eq:multi}, so they
can be ignored.  But, from the equation $\sum_{\mathbf{s}_i\in
  S(p\mathbf{1})} r_i\mathbf{s}_i=p\mathbf{1}$, the case $d=p$ happens
precisely when:
\begin{itemize}
  \item
    there are $k$ unit vectors $\mathbf{s}_1,\ldots,\mathbf{s}_k\in
    U$, for $1\le k\le n$, each of whose associated multiplicity is $p$, as well as the zero vector $\mathbf{0}$, whose multiplicity is $p(n-k)$, such that $\mathbf{s}_1+\cdots+\mathbf{s}_k+\mathbf{0}=\mathbf{1}$. 
\end{itemize}
\end{proof}
\begin{example}\label{example:ms}
  When $N=1$, then $U=\set{1}$. Hence, \\
  $\binom{pn}{p}_f \equiv \binom{pn}{p}f(0)^{p(n-1)}f(1)^p\pmod{pn}$ because only the term $k=1$ leads to a valid solution, since $1$ cannot be the sum of two or more elements from $U$. When $N=2$, then $U=\set{(0,1),(1,0),(1,1)}$ and the relevant terms are $k=1,2$. The formula becomes 
\begin{align*}
  \binom{pn}{p(1,1)}_f \equiv & \binom{pn}{p}f((0,0))^{p(n-1)}f((1,1))^p\\
  &+\frac{(pn)!}{(p!)^2(p(n-2))!}f((0,0))^{p(n-2)}f((0,1))^pf((1,0))^p\pmod{pn}. 
\end{align*}
\end{example}

Recall that the ordinary binomial coefficients satisfy Lucas'
theorem, namely,
\begin{align*}
  \binom{k}{n} \equiv \prod\binom{k_i}{n_i}\pmod{p},
\end{align*}
whenever $k=\sum k_ip^i$ and $n=\sum n_ip^i$ with $0\le
n_i,k_i<p$. 
Bollinger and
Burchard \cite{Bollinger:1990} generalize this to extended binomial coefficients. 
We
further generalize 
to weighted vector compositions. %building on a %an analogous 
\begin{theorem}[Lucas' theorem]\label{theorem:lucas}
  Let $p$ be prime and let %$n=\sum_{i=0}^t n_ip^i$ and 
  $k=\sum_{j=0}^r k_jp^j$, where $0\le k_j<p$ for $j=0,\ldots,r$. 
  Let $\Bell\in\nn^N$. 
  Then
  \begin{align*}
   \binom{k}{\Bell}_{f} \equiv \sum_{(\mathbf{m}_0,\ldots,\mathbf{m}_r)}\prod_{i=0}^r \binom{k_i}{\mathbf{m}_i}_{f}\pmod{p},
  \end{align*}
  whereby the sum is over all $(\mathbf{m}_0,\ldots,\mathbf{m}_r)$ that satisfy
  $\mathbf{m}_0+\mathbf{m}_1p+\cdots+\mathbf{m}_rp^r=\Bell$. %, $0\le s_i\le nn_i$. 
\end{theorem}
\begin{proof} We have
  \begin{align*}
    \sum_{\Bell\in\nn^N}\binom{k}{\Bell}_f\mathbf{x}^{\Bell} &= \left(\sum_{\mathbf{s}\in\nn^N}f(\mathbf{s})\mathbf{x}^{\mathbf{s}}\right)^k = \prod_{j=0}^r\left(\sum_{\mathbf{s}\in\nn^N}f(\mathbf{s})\mathbf{x}^{\mathbf{s}}\right)^{k_jp^j}\\
    &= \prod_{j=0}^r\left(\sum_{\mathbf{s}\in\nn^N}\binom{p^j}{\mathbf{s}}_f\mathbf{x}^{\mathbf{s}}\right)^{k_j} \\
    &\equiv \prod_{j=0}^r \left(\sum_{\mathbf{m}\in\nn^N}f(\mathbf{m})\mathbf{x}^{p^j\mathbf{m}}\right)^{k_j}
    = \prod_{j=0}^r\left(\sum_{\mathbf{m}\in\nn^N}\binom{k_j}{\mathbf{m}}_f\mathbf{x}^{p^j\mathbf{m}}\right) \\
    &= \sum_{\Bell\in\nn^N}\left(\sum_{(\mathbf{m}_0,\ldots,\mathbf{m}_r)}\binom{k_0}{\mathbf{m}_0}_f\cdots\binom{k_r}{\mathbf{m}_r}_f\right)\mathbf{x}^{\Bell}\pmod{p},
  \end{align*}
  where the fourth relation (congruence) follows from Theorem \ref{theorem:prime2},
  and the theorem follows by comparing the coefficients of $\mathbf{x}^{\Bell}$. 
\end{proof}
\begin{example}\label{example:8}
  For a similar situation as in Example \ref{example:170}, let $p=3$ and $k=5=2+1\cdot p$. Thus, $(k_0,k_1)=(2,1)$. For $\Bell=(3,6)$, the relevant $(\mathbf{m}_0,\mathbf{m}_1)$ such that $\Bell=\mathbf{m}_0+p\mathbf{m}_1$ are:
  \begin{align*}
    (3,6) = (0,0)+3(1,2) = (0,3)+3(1,1)=(3,3)+3(0,1)=(0,6)+3(1,0). 
  \end{align*}
  No other $\mathbf{m}_1$ must be looked at, because $k_1=1$ and $\binom{1}{\mathbf{x}}_f=f(\mathbf{x})$ and the specified $f$ is zero outside $\set{(0,1),(1,0),(1,1),(1,2),(2,1)}$. Hence:
  \begin{align*}
    \binom{5}{(3,6)}_f &\equiv  \binom{2}{(0,0)}_f\cdot \binom{1}{(1,2)}_f+ \binom{2}{(0,3)}_f\cdot \binom{1}{(1,1)}_f\\
  &+\binom{2}{(3,3)}_f\cdot \binom{1}{(0,1)}_f
  + \binom{2}{(0,6)}_f\cdot \binom{1}{(1,0)}_f \\
    &= 0\cdot 1+0\cdot 1+2\cdot 1+0\cdot 1 = 2 \pmod{p},
  \end{align*}
  by Theorem \ref{theorem:lucas}, which is true, since $\binom{5}{(3,6)}_f=80$. 
\end{example}

Our final result in this section allows a fast computation of the
coefficients $\binom{k}{\Bell}_f$ modulo a prime $p$. See Granville
\cite{Granville:1997} for the corresponding result for the special case of
ordinary binomial coefficients. 
\begin{theorem}\label{theorem:13}
Let $p$ be prime, $k\ge 0$, $\Bell\in\nn^N$. Then, 
\begin{align*}
\binom{k}{\Bell}_{f} \equiv \sum_{\mathbf{m}\in\nn^N}\binom{k_1}{\mathbf{x} -\mathbf{m}}_{f}\binom{k_0}{\Bell_0+\mathbf{m}p}_f\pmod{p},
\end{align*}
whereby $k=k_0+k_1p$ with $0\le k_0<p$, and $\Bell=\Bell_0+\mathbf{x}p$, where each component $\ell$ of $\Bell_0$ satisfies $0\le \ell<p$. 
\end{theorem}
\begin{proof}
  We have
  \begin{align*}
    \left(\sum_{\mathbf{m}\in\nn^N}f(\mathbf{m})\mathbf{x}^{\mathbf{m}}\right)^p \equiv \sum_{\mathbf{m}\in\nn^N}
    f(\mathbf{m})\mathbf{x}^{p\mathbf{m}} \pmod{p}
  \end{align*}
  by Theorem \ref{theorem:prime1} and therefore, with $k=k_0+k_1p$,
  for $0\le k_0<p$, 
  \begin{align*}
    \left(\sum_{\mathbf{s}}f(\mathbf{s})\mathbf{x}^{\mathbf{s}}\right)^{k_0+k_1p} &\equiv
    \left(\sum_{\mathbf{t}}f(\mathbf{t})\mathbf{x}^{\mathbf{t}}\right)^{k_0} \left(\sum_{\mathbf{s}}
    f(\mathbf{s})\mathbf{x}^{p\mathbf{s}}\right)^{\lfloor k/p \rfloor} \\
    &=
    \sum_{\mathbf{t},\mathbf{s}}\binom{k_0}{\mathbf{t}}_f\binom{\lfloor
      k/p\rfloor}{\mathbf{s}}_f\mathbf{x}^{p\mathbf{s}+\mathbf{t}}
    \pmod{p}.
  \end{align*}
  Now, since $\binom{k}{\Bell}_f$ is the coefficient of $\mathbf{x}^{\Bell}$ of
  $\left(\sum_{\mathbf{s}}f(\mathbf{s})\mathbf{x}^{\mathbf{s}}\right)^{k_0+k_1p}$, we have 
  \begin{align*}
    \binom{k}{\Bell}_f \equiv \sum_{p\mathbf{s}+\mathbf{t}=\Bell}\binom{\lfloor
      k/p\rfloor}{\mathbf{s}}_f\binom{k_0}{\mathbf{t}}_f\pmod{p},
  \end{align*}
  and the theorem follows after re-indexing the summation on the RHS. 
\end{proof}
\begin{example}
  In the situation of Example \ref{example:8}, consider $p=3$,
  $\Bell=(3,6)=(0,0)+(1,2)\cdot p$ and $k=5=2+1\cdot p$. 
  By Theorem \ref{theorem:13}, we have hence to consider sums of
  products of the form 
  \begin{align*}
    \binom{1}{(1,2)-\mathbf{m}}_f\cdot\binom{2}{(0,0)+p\mathbf{m}}_f=f((1,2)-\mathbf{m})\cdot\binom{2}{p\mathbf{m}}_f. 
  \end{align*}
  Since $f$ is zero outside of $\set{(0,1),(1,0),(1,1),(1,2),(2,1)}$, $\mathbf{m}$ ranges over \\
$\set{(0,0),(0,1),(1,1),(0,2)}$ and the summation is the same as in Example \ref{example:8}.

For a more challenging example, let $p=7$, $\Bell=(5,9)=(5,2)+(0,1)p$
and $k=8=1+1\cdot p$. Here, we have to consider sums of products of
the form 
\begin{align*}
  \binom{1}{(0,1)-\mathbf{m}}_f\cdot\binom{1}{(5,2)+p\mathbf{m}}_f=f((0,1)-\mathbf{m})\cdot
  f{((5,2)+p\mathbf{m})}. 
\end{align*}
Due to the specification of $f$, the only possible such term ($\mathbf{m}=(0,0)$) leads to the sum value of $0$. Indeed, $\binom{8}{(5,9)}_f=4368=7\cdot 2^4\cdot 39$.  
\end{example}

\section{Congruences and identities for sums of $\binom{k}{\Bell}_f$}\label{sec:congruence2}
In this section, we consider divisibility properties and identities for 
sums of $\binom{k}{\Bell}_f$. 
First, we focus on the number
$c_f(\Bell)=\sum_{k\ge 0}\binom{k}{\Bell}_f$ of vector compositions with 
arbitrary number of parts.
In Theorems \ref{theorem:glashier} and
\ref{theorem:rowsum}, 
we then investigate  
particular divisibility properties 
for 
the total number of all $f$-weighted vector compositions of $\Bell$ 
where $\Bell$ ranges over particular sets  
$L$ and where the number of parts is fixed, that is, 
we evaluate divisibility of $\sum_{\Bell\in L}\binom{k}{\Bell}_f$. 
We also generalize the notion of $s$-color compositions
\cite{Agarwal:2000} in this section and derive a corresponding
identity.   

At first, 
we establish that $c_f(\Bell)$ 
satisfies a weighted linear
recurrence where the weights are given by $f$. 
\begin{theorem}\label{theorem:recurrence}
  For $\Bell\in\nn^N$, ${\Bell} \neq \mathbf{0}$, we have that
  \begin{align*}
    c_f(\Bell) = \sum_{\mathbf{m}\in\nn^N} f(\mathbf{m})c_f({\Bell}-\mathbf{m}),
  \end{align*}
  %where $c_f(0)=1$, $c_f(n)=0$,
  where we define $c_f(\mathbf{0})=1$ and $c_f(\Bell)=0$ if any component $\ell$ of $\Bell$ is smaller than zero.  
\end{theorem}
\begin{proof}
  An $f$-weighted vector composition
  $[\mathbf{m}_1,\ldots,\mathbf{m}_{k-1},\mathbf{m}_k]$ of $\Bell$
  ends, in its last 
  part, with exactly one of the values $\mathbf{m}=\mathbf{m}_k\in\nn^N$, and $\mathbf{m}$ may be
  colored in $f(\mathbf{m})$ different colors. %Then, $\Bell-\mathbf{m}$
  Moreover, $[\mathbf{m}_1,\ldots,\mathbf{m}_{k-1}]$ is a vector
  composition of $\Bell-\mathbf{m}$. 
\end{proof}
Before investigating divisibility of $c_f(\Bell)$, we detail special cases of $c_f(\Bell)$ that arise for particular $f$. 
\begin{example}
  When $f_D$ is the indicator function on the set \\
  $D=\set{(0,1),(1,0),(1,1)}$, then $c_{f_D}(\Bell)=c_{f_D}(m,n)$ is the well-known Delannoy sequence \cite{Banderier:2004}, which counts the number of lattice paths from $(0,0)$ to $(m,n)$ with steps in $D$ (i.e., east, north, north-east). 
The underlying lattice paths are of interest in sequence alignment problems in computational biology and computational linguistics.
%, as mentioned earlier. 
They also appear in so-called edit distance problems \cite{Levenshtein:1966} in which the minimal number of insertions and deletions is sought that transforms one sequence into another.  
Closed-form expressions for the Delannoy numbers are 
\begin{align*}
  c_{f_D}(m,n) = \sum_{d=0}^m 2^d\binom{m}{d}\binom{n}{d}=\sum_{d=0}^n\binom{n}{d}\binom{m+n-d}{n}.
\end{align*}
The weighted Delannoy numbers \cite{Razpet:2002}, for which $f_{\text{WD}}((1,0))=a$, \\ $f_{\text{WD}}((0,1))=b$ and $f_{\text{WD}}((1,1))=c$, for integers $a,b,c\ge 1$, have closed-form expression
\begin{align*}
  c_{f_{\text{WD}}}(m,n) = a^mb^n\sum_{d\ge 0}\binom{m}{d}\binom{n}{d}\Bigl(\frac{ab+c}{ab}\Bigr)^d. 
\end{align*}
When $f_W$ is the indicator function on the set $W=\set{(1,1),(1,2),(2,1),(2,2)}$, then $c_{f_W}(\Bell)=c_{f_W}(m,n)$ are known as Whitney numbers \cite{Bona:2010}. The diagonals are listed as integer sequence A051286. A closed-form expression can be derived as 
\begin{align*}
  c_{f_W}(m,n) = \sum_{k\ge 0}\binom{m-k}{k}\binom{n-k}{k}.
\end{align*}
The diagonals of $c_{f_M}$, where $M=\set{(1,1),(1,2),(2,1)}$, are listed as integer sequence A098479. The diagonals of $c_{f_R}$, where $R=\set{(x,y)\sd x\ge 1,y\ge 0}$, are listed as integer sequence A047781. The diagonals of $c_{f_A}$, where $A=\set{(x,y,z)\sd 0\le x,y,z\le 1}-\set{\mathbf{0}_3}$, are listed as integer sequence A126086. They appear in alignment problems of multiple (in this case, three) sequences. The case of $c_{f_S}$, for $S=\nn^N-\set{\mathbf{0}}$, %is 
counts the original vector compositions considered in \cite{Andrews:1975}. A closed-form expression is given by 
\begin{align*}
  c_{f_S}(\ell_1,\ldots,\ell_N)=\sum_{k=0}^{\ell_1+\cdots+\ell_N}\sum_{i=0}^k(-1)^i\binom{k}{i}\prod_{j=1}^N\binom{\ell_j+k-i-1}{\ell_j}.
\end{align*} 
It has been noted that $c_{f_S}(\ell,\ldots,\ell)=2^{\ell-1}c_{f_U}(\ell,\ldots,\ell)$, where \\
 $U=\set{(s_1,\ldots,s_N)\sd s_i\in\set{0,1}}-\set{\mathbf{0}}$ \cite{Duchi:2004}. The latter numbers generalize the Delannoy numbers and admit the closed-form expression \cite{Slowinski:1998} 
\begin{align*}
  c_{f_U}(\ell_1,\ldots,\ell_N) = \sum_{k=\max\set{\ell_1,\ldots,\ell_N}}^{\ell_1+\cdots+\ell_N}\sum_{i=0}^k(-1)^i\binom{k}{i}\prod_{j=1}^N\binom{k-i}{\ell_j}. 
\end{align*}
\end{example}

Next, we generalize the concept of $s$-color compositions for ordinary colored compositions, for which the weighting function is $f(s)=s$ for each part size $s$, to weighted vector compositions. 
Of course, there are many possible extensions of the concept of $s$-color integer compositions to vector compositions. The most natural is  probably the following:
\begin{definition}
  We call an $f$-weighted vector composition of $\Bell$ an \emph{$\mathbf{s}$-color composition} when 
  \begin{align*}
    f(\mathbf{s})=s_1\cdots s_N
  \end{align*}
  for all $\mathbf{s}\in\nn^N$.
\end{definition}
This definition %allows to independently label 
inherently captures an independent labeling of the vector components $s_1,\ldots,s_N$ into $s_1$ colors (for component $1$),$\ldots$, $s_N$ colors (for component $N$). It is well-known that ordinary $s$-color compositions \cite{Agarwal:2000} are closely related to ``1-2-color compositions'', that is, integer compositions that only have part sizes in $\set{1,2}$; see, e.g., Shapcott \cite{Shapcott:2013}. The next theorem generalizes this relationship.  
\begin{theorem}\label{theorem:fprod}
  Let $f_{\text{prod}}(\mathbf{s})=s_1\cdots s_N$ for all $\mathbf{s}\in\nn^N$ and let $g$ be the indicator function on $S=\set{(s_1,\ldots,s_N)\sd s_i\in\set{1,2}}$. Then
  \begin{align*}
    c_{f_{\text{prod}}}(\ell\mathbf{1}) = c_{g}((2\ell-1)\mathbf{1}).
  \end{align*}
for all $\ell>0$.
\end{theorem}
\begin{proof}[Proof sketch]
  Let $(s_1,\ldots,s_N)^{1},\ldots,(s_1,\ldots,s_N)^{s_1\cdots s_N}$ be the $s_1\cdots s_N$ colorations of part size $(s_1,\ldots,s_N)$. We bijectively re-write them to individual components $(s_1^1,\ldots,s_N^1),\ldots,(s_1^{s_1},\ldots,s_N^{s_N})$. Now, when we have a sum $\mathbf{s}^r+\mathbf{t}^q=\ell\mathbf{1}$ (and similarly for more than two terms) this reads in components
\begin{align*}
  \begin{pmatrix}s_1^{r_1}\\ \vdots \\ s_N^{r_N}\end{pmatrix}+
  \begin{pmatrix}t_1^{q_1}\\ \vdots \\ t_N^{q_N}\end{pmatrix}=
  \begin{pmatrix}\ell\\ \vdots \\ \ell\end{pmatrix}
\end{align*}
where $r_1,\ldots,r_N$ and $q_1,\ldots,q_N$ denote the bijective re-writings. Consider this equation in each row, $s_i^{r_i}+t_i^{q_i}=\ell$. Encode the integer composition $(s_i^{r_i},t_i^{q_i})$ of $\ell$ into the ``cross-and-dash representation'' of Shapcott \cite{Shapcott:2013} in which crosses separate parts and a part value of size $\pi$ with color $1\le c\le \pi$ is denoted by $\pi-1$ dashes and one cross in position $c$. Then, as in  Shapcott \cite{Shapcott:2013}, Proposition 2, let crosses stand for 1s and dashes for 2s. This proves the bijection between $f_{\text{prod}}$-weighted compositions and $g$-weighted compositions. 
 
The table below illustrates the $\mathbf{s}$-color compositions of $(3,3)$ (into two parts) and the uniquely corresponding $g$-weighted compositions (into four parts).
\begin{table}[!htb]
  \centering
  \begin{tabular}{|ccc|}\hline
    $\mathbf{s}$-color & cross-and-dash & 1-2 compositions\\ \hline
    $\begin{pmatrix}2^1\\ 2^1\end{pmatrix}+\begin{pmatrix}1^1 \\ 1^1\end{pmatrix}$ & $\begin{pmatrix}\times{-}\times\times\\ \times{-}\times\times\end{pmatrix}$ & $\begin{pmatrix} 1 & 2 & 1 & 1 \\ 1 & 2 & 1 & 1\end{pmatrix}$\\
          $\begin{pmatrix}2^1\\ 2^2\end{pmatrix}+\begin{pmatrix}1^1 \\ 1^1\end{pmatrix}$ & $\begin{pmatrix}\times{-}\times\times\\ {-}\times\times\times\end{pmatrix}$ & $\begin{pmatrix} 1 & 2 & 1 & 1 \\ 2 & 1 & 1 & 1\end{pmatrix}$\\
     $\begin{pmatrix}2^2\\ 2^1\end{pmatrix}+\begin{pmatrix}1^1 \\ 1^1\end{pmatrix}$ & $\begin{pmatrix}{-}\times\times\times\\ \times{-}\times\times\end{pmatrix}$ & $\begin{pmatrix} 2 & 1 & 1 & 1 \\ 1 & 2 & 1 & 1\end{pmatrix}$\\
    $\begin{pmatrix}2^2\\ 2^2\end{pmatrix}+\begin{pmatrix}1^1 \\ 1^1\end{pmatrix}$ & $\begin{pmatrix}{-}\times\times\times\\ {-}\times\times\times\end{pmatrix}$ & $\begin{pmatrix} 2 & 1 & 1 & 1 \\ 2 & 1 & 1 & 1\end{pmatrix}$\\
    %$(1,1)+(2,2)$ & & \\
     $\begin{pmatrix}2^1\\ 1^1\end{pmatrix}+\begin{pmatrix}1^1 \\ 2^1\end{pmatrix}$ & $\begin{pmatrix}\times{-}\times\times\\ \times\times\times{-}\end{pmatrix}$ & $\begin{pmatrix} 1 & 2 & 1 & 1 \\ 1 & 1 & 1 & 2\end{pmatrix}$\\
    $\begin{pmatrix}2^2\\ 1^1\end{pmatrix}+\begin{pmatrix}1^1 \\ 2^1\end{pmatrix}$ & $\begin{pmatrix}{-}\times\times\times\\ \times\times\times{-}\end{pmatrix}$ & $\begin{pmatrix} 2 & 1 & 1 & 1 \\ 1 & 1 & 1 & 2\end{pmatrix}$\\
    $\begin{pmatrix}2^1\\ 1^1\end{pmatrix}+\begin{pmatrix}1^1 \\ 2^2\end{pmatrix}$ & $\begin{pmatrix}\times{-}\times\times\\ \times\times{-}\times\end{pmatrix}$ & $\begin{pmatrix} 1 & 2 & 1 & 1 \\ 1 & 1 & 2 & 1\end{pmatrix}$\\
    $\begin{pmatrix}2^2\\ 1^1\end{pmatrix}+\begin{pmatrix}1^1 \\ 2^2\end{pmatrix}$ & $\begin{pmatrix}{-}\times\times\times\\ \times\times{-}\times\end{pmatrix}$ & $\begin{pmatrix} 2 & 1 & 1 & 1 \\ 1 & 1 & 2 & 1\end{pmatrix}$\\
    %$(1,2)+(2,1)$ \\
    \hline
  \end{tabular}
\end{table}
The table omits the further eight cases corresponding to $(1,1)+(2,2)$ and $(1,2)+(2,1)$. 
\end{proof}
\begin{example}
  The number of $f_{\text{prod}}$-weighted vector compositions of $(\ell,\ell)$ are given by the integer sequence
  \begin{align*}
    1,5,26,153,931,5794,36631,234205,\ldots
  \end{align*}
  for $\ell=1,2,3,\ldots$. 
  The number of $g$-weighted vector compositions of $(\ell,\ell)$ are given by integer sequence A051286
  \begin{align*}
    1,2,5,11,26,63,153,376,931,2317,5794,14545,36631,92512,234205,\ldots
  \end{align*}
\end{example}

When $f$ is arbitrary but zero almost everywhere, that is,
$f(\mathbf{x})\neq 0$ for only finitely many $\mathbf{x}$, 
then $c_f(\Bell)$ satisfies a %n $m$-th order
linear recurrence by Theorem 
\ref{theorem:recurrence}. 
When $N=1$, that is, vectors $\Bell$ are one-dimensional, then $c_f$ satisfies an $m$-th order linear recurrence of the form
\begin{align*}
  c_f(n+m) = f(1)c_f(n+m-1)+\cdots+f(m)c_f(n) 
\end{align*}
in this situation.
  
For such sequences, Somer \cite{Somer:1987} specifies varying congruence relationships, one of which translates to the following result in our context. 
\begin{theorem}[\cite{Eger:2016}, Theorem 27]\label{theorem:recgeneral}
  Let $p$ be a prime and let $b$ a nonnegative integer. Let
  $f:\nn\goesto\nn$ be zero almost everywhere, i.e., $f(x)=0$ for all
  $x>m$ for some positive $m$. Then
  \begin{align*}
    c_f(n+mp^b)\equiv&
    f(1)c_f(n+(m-1)p^b)+f(2)c_f(n+(m-2)p^b)+\cdots\\
    &+f(m)c_f(n)\pmod{p}. 
  \end{align*}
  \qed
\end{theorem}
However, when $N>1$, these results are not applicable. 
One possibility would be to project vectors in $\nn^N$ onto $\nn$
via a bijection 
$\tau:\nn\goesto\nn^N$ 
and then define new quantities $\tilde{c}_f$
\begin{align*}
  \tilde{c}_f(n) = c_f(\tau(n))
\end{align*}
for which the findings of \cite{Somer:1987} and others might be applicable. 
The problem with such a specification is that the bijection does not lead, in general, to fixed order linear recurrences because $\tau$ can map different $n,n'$ to `arbitrary' points in $\nn^N$, so that e.g.\ $\tilde{c}_f(100)$ may be a function of $\tilde{c}_f(90)$ and $\tilde{c}_f(80)$, but $\tilde{c}_f(1000)$ may be a function of $\tilde{c}_f(543)$ and $\tilde{c}_f(389)$.  

Another result, for $N=2$, is the following. 
Consider the weighted Delannoy numbers for which $f_{\text{WD}}((1,0))=a$, 
 $f_{\text{WD}}((0,1))=b$ and $f_{\text{WD}}((1,1))=c$ as above. Razpet \cite{Razpet:2002} shows that these numbers satisfy a `Lucas property'.
\begin{theorem}[\cite{Razpet:2002}, Theorem 2]
  Let $p$ be prime, $n\ge 1$, and let integers $a_k,b_k$ satisfy
  \begin{align*}
    0\le a_k,b_k<p,\quad \text{for all } k=0,1,\ldots,n.
  \end{align*}
  Then
  \begin{align*}
  & c_{f_{\text{WD}}}(a_np^n+\cdots+a_1p+a_0\:,\: b_np^n+\cdots+b_1p+b_0)
  \\
  \equiv & c_{f_{\text{WD}}}(a_n,b_n)\cdots c_{f_{\text{WD}}}(a_1,b_1)c_{f_{\text{WD}}}(a_0, b_0)\pmod{p}.
  \end{align*}
  \qed
\end{theorem}
Finally, we consider the number of $f$-weighted vector compositions, with
fixed number of parts, of \emph{all} vectors $\Bell$ in some particular
sets $L$. Introduce the following notation: 
\begin{align*}
	{k\brack \mathbf{r}}_{\mathbf{m},f} = \sum_{\set{\Bell\in\nn^N\sd\Bell=\mathbf{Am}+\mathbf{r}, \text{ for some } \mathbf{A}\in\mathcal{D}(N)}}\binom{k}{\Bell}_{f},
\end{align*}
where $\mathcal{D}(N)$ is the set of $N\times N$ diagonal matrices with nonnegative integer entries. %, i.e., $\mathcal{D}=\set{}$
Note that ${k\brack \mathbf{r}}_{\mathbf{m},f}$ generalizes the binomial sum
notation (cf.\ \cite{Sun:2007}). 
By the Vandermonde convolution,
${k\brack \mathbf{r}}_{\mathbf{m},f}$ satisfies  
\begin{equation}
  \label{eq:vand_sum}
  \begin{split}
  {k \brack \mathbf{r}}_{\mathbf{m},f} &= 
  \sum_{\set{\Bell\in\nn^N\sd\Bell=\mathbf{A}\mathbf{m}+\mathbf{r}, \text{ for some } \mathbf{A}\in\mathcal{D}(N)}}\binom{k}{\Bell}_{f} = \sum_{\mathbf{A}\in\mathcal{D}(N)}\binom{k}{\mathbf{A}\mathbf{m}+\mathbf{r}}_f 
\\ 
  &= \sum_{\mathbf{A}\in\mathcal{D}(N)}\sum_{\mathbf{s}\in\nn^N}f(\mathbf{s})\binom{k-1}{\mathbf{A}\mathbf{m}+\mathbf{r}-\mathbf{s}}_f
  \\ &=  \sum_{\mathbf{s}\in\nn^N}f(\mathbf{s})\sum_{\mathbf{A}\in\mathcal{D}(N)}\binom{k-1}{\mathbf{Am}+\mathbf{r}-\mathbf{s}}_f 
 = 
\sum_{\mathbf{s}\in\nn^N} f(\mathbf{s}){k-1\brack \mathbf{r}-\mathbf{s}}_{\mathbf{m},f}.
\end{split}
\end{equation}

Our first theorem in this context goes back to J.\ W.\ L.\ Glaisher, 
and 
its proof is 
inspired by the corresponding proof for binomial sums due to Sun
(cf.\ \cite{Sun:2007}, and references therein). 
\begin{theorem}\label{theorem:glashier}
Let $\mathbf{m}=(m_1,\ldots,m_N)\in\nn^N$. 
For any prime $p\equiv 1\pmod{m_i}$, for all $i=1,\ldots,N$, and any $k\ge 1$, $\mathbf{r}\in\nn^N$,  
\begin{align*}
{k+p-1\brack \mathbf{r}}_{\mathbf{m},f}\equiv {k\brack \mathbf{r}}_{\mathbf{m},f}\pmod{p}.
\end{align*}
\end{theorem}
\begin{proof}
For $k=1$, 
\begin{align*}
{p \brack \mathbf{r}}_{\mathbf{m},f} &= \sum_{\set{\Bell\in\nn^N\sd\Bell=\mathbf{Am}+\mathbf{r}}} \binom{p}{\Bell}_f\\ &\equiv \sum_{\set{\Bell\in\nn^N\sd \Bell=\mathbf{Am}+\mathbf{r}=p\mathbf{q}, \text{ for some } \mathbf{q}}}\binom{p}{\Bell}_f\pmod{p}. %\equiv \sum_{q\ge 0,q\equiv r\pmod{m}}f(q)\pmod{p},
\end{align*}
by Theorem \ref{theorem:prime1}. Now, in components, the equation $\mathbf{Am}+\mathbf{r}=p\mathbf{q}$ means that $a_{ii}m_i+r_i=pq_i$. Since $p\equiv 1\pmod{m_i}$, we have $q_i\equiv r_i\pmod{m_i}$, i.e., $q_i = c_im_i+r_i$. In vector notation this means $\mathbf{q}=\mathbf{Cm}+\mathbf{r}$ for the diagonal matrix $\mathbf{C}$ with entries $C_{ii}=c_i$. Therefore,
\begin{align*}
{p \brack \mathbf{r}}_{\mathbf{m},f} \equiv \sum_{\set{\mathbf{q}\in\nn^N\sd \mathbf{q}=\mathbf{Cm}+\mathbf{r}}} \binom{p}{p\mathbf{q}}_f\equiv 
\sum_{\set{\mathbf{q}\in\nn^N\sd \mathbf{q}=\mathbf{Cm}+\mathbf{r}}} \binom{1}{\mathbf{q}}\pmod{p} 
\end{align*}
using Theorem \ref{theorem:special}. 
The RHS is %congruent 
${1 \brack \mathbf{r}}_{\mathbf{m},f}$. % modulo $p$. 
For $k>1$, the result follows by induction using \eqref{eq:vand_sum}. 
\end{proof}
\begin{example}
  Let $f$ be the indicator function on the set $\set{(0,1),(1,1),(1,1)}$. Let $p=5$, $k=2$, $\mathbf{m}=(4,1)$ and $\mathbf{r}=(1,0)$. To evaluate ${k \brack \mathbf{r}}_{\mathbf{m},f}$, we consider all matrices $\mathbf{A}\in\mathcal{D}(2)$ and all corresponding sums $\mathbf{Am}+\mathbf{r}$. Since it is impossible to write $m_1=4$ (or larger) as the sum of $k=2$ numbers in $\set{0,1}$, $a_{11}$ must be zero. The only suitable matrices are then 
\begin{align*}
  \mathbf{A}_0=\begin{pmatrix} 0 & 0 \\ 0 & 1\end{pmatrix},\quad
  \mathbf{A}_1=\begin{pmatrix} 0 & 0 \\ 0 & 2\end{pmatrix}.
\end{align*}
The corresponding values $\mathbf{Am}+\mathbf{r}$ are
\begin{align*}
  \Bell_0 = (1,1),\quad \Bell_1 = (1,2).
\end{align*}
We easily find that $\binom{k}{\Bell_0}_f=\binom{k}{\Bell_1}_f=2$ and therefore ${k \brack \mathbf{r}}_{\mathbf{m},f}=4$. Similarly, for ${k+p-1 \brack \mathbf{r}}_{\mathbf{m},f}$, we have to evaluate matrices
\begin{align*}
  \mathbf{A}_{n,0}=\begin{pmatrix} 0 & 0 \\ 0 & n\end{pmatrix},\quad\quad\text{and},\quad 
  \mathbf{A}_{n,1}=\begin{pmatrix} 1 & 0 \\ 0 & n\end{pmatrix}.
\end{align*}
and correspondingly
\begin{align*}
  \binom{6}{(1,n)}_f,\quad\text{and},\quad\binom{6}{(5,n)}_f
\end{align*}
for $n=0,\ldots,6$. Summing up yields ${k+p-1 \brack \mathbf{r}}_{\mathbf{m},f}=204$, which is indeed $\equiv 4\pmod{5}$. 
\end{example}

\begin{theorem}\label{theorem:rowsum}
Let $f(\mathbf{s})={0}$ for almost all $\mathbf{s}\in\nn^N$. 
Consider ${k\brack \mathbf{0}}_{\mathbf{1},f}$, the row sum in row $k\ge 0$, or,
equivalently, the 
total number of $f$-weighted vector 
                  compositions with $k$ parts.   
Let $M=\sum_{\mathbf{s}\in\nn^N}f(\mathbf{s})$. Then
\begin{align*}
  {k\brack \mathbf{0}}_{\mathbf{1},f} = M^k. %\pmod{2}
\end{align*}
for all $k>0$. 
This implies the congruences
\begin{align*}
  {k\brack \mathbf{0}}_{\mathbf{1},f} \equiv M\pmod{2},\quad\text{and},\quad{k\brack \mathbf{0}}_{\mathbf{1},f} \equiv M^{a_0+\cdots+a_r}\pmod{p},
\end{align*}
for any prime $p$ by Fermat's little theorem, where $k=a_0+\cdots+a_rp^r$, with $0\le a_i<p$ for all $i=0,\ldots,r$. 
\end{theorem}
\begin{proof}
  Consider the equation $(\sum_{\mathbf{s}\in\nn^N}f(\mathbf{s})\mathbf{x}^{\mathbf{s}})^k=\sum_{\Bell\in\nn^N}\binom{k}{\Bell}_{f}\mathbf{x}^{\Bell}$. 
  Plug in $\mathbf{x}=\mathbf{1}\in\nn^N$.
\end{proof}
\begin{remark}
  Note that the previous theorem generalizes the fact that the number
  of odd entries in row $k$ in Pascal's triangle is a multiple of $2$.
\end{remark}
\begin{example}\label{example:sum}
  When $f$ is the indicator function on $\set{(0,1),(1,1),(1,1)}$ 
  then $M=3$ and so the row sum in row $k>0$ is $3^k$ and thus always odd. To illustrate, for $k=1$, we have $\binom{k}{(0,1)}_f=\binom{k}{(1,1)}_f=\binom{k}{(1,1)}_f=1$, so their sum is $3$. For $k=2$, we have to consider all $\Bell=(x,y)$ with $x,y\le 2$. We find for all $\Bell$ such that 
$\binom{2}{\Bell}_f$ is non-zero:
  \begin{align*}
    &\binom{k}{(1,1)}_f = 2,\:\: \binom{k}{(2,0)}_f = \binom{k}{(0,2)}_f = 1,\:\: \binom{k}{(1,2)}_f = \binom{k}{(2,1)}_f = 2,%\:\: 
   \\ &\binom{k}{(2,2)}_f = 1.
  \end{align*}
  Hence, their sum is indeed $9$.
\end{example}

\section{Asymptotics of $c_f(\Bell)$}\label{sec:asymp}
We can find asymptotics of $c_f(\Bell)$ by looking at its multivariate generating function 
\begin{align*}
  F(\mathbf{x})=\sum_{\Bell\in\nn^N}c_f(\Bell)\mathbf{x}^{\Bell}=\sum_{k\ge 0}\left(\sum_{\mathbf{s}\in\nn^N}f(\mathbf{s})\mathbf{x}^{\mathbf{s}}\right)^k = \frac{1}{1-\sum_{\mathbf{s}\in\nn^N}f(\mathbf{s})\mathbf{x}^{\mathbf{s}}}. 
\end{align*}
While methods for determining the asymptotic growth of the coefficients of a generating function in one variable are well-established \cite{Flajolet:2009}, 
methods for generating functions of several variables are less ubiquitous.
However, \cite{Raichev:2007} and \cite{Pemantle:2008} discuss such cases. Particularly simple results obtain when $J(\mathbf{x}):=1-\sum_{\mathbf{s}\in\nn^N}f(\mathbf{s})\mathbf{x}^{\mathbf{s}}$ is symmetric in $\mathbf{x}$. 

For instance, \cite{Raichev:2007} discuss the case when $f$ is the indicator function on $\set{(1,0),(0,1),(1,1)}$, so that $J(x,y)=1-x-y-xy$. They determine the set of ``critical points'', that is, the points $(x_0,y_0)$ that satisfy $J(x_0,y_0)=0$ and $x_0\frac{\partial J(x_0,y_0)}{\partial x}=y_0\frac{\partial J(x_0,y_0)}{\partial x}$ in the positive orthant. They find that $(x_0,y_0)=(L-1,L-1)$, where $L=\sqrt{2}$ is the only solution, from which follows the asymptotic
\begin{align*}
  c_f(\ell,\ell)\sim x_0^{-\ell}y_0^{-\ell}\sqrt{\frac{1}{L(2-L)^22\pi\ell}}
\end{align*} 
using their Theorems 3.2 and 3.3. More general cases such as when $f$ is the indicator function on $\set{\mathbf{x}\neq\mathbf{0}\in\nn^N\sd x_i\in\set{0,1}}$ or on $\set{\mathbf{x}\in\nn^N\sd x_i\in\set{1,2}}$ can be solved analogously, but require more work to find the critical points and the implied asymptotics.  

\begin{theorem}[\cite{Griggs:1990}, Theorem 2]
  Let $f$ be the indicator function on $S=\set{(s_1,\ldots,s_N)\sd s_i\in\set{0,1}}-\set{\mathbf{0}}$. Then
  \begin{align*}
    c_{f}(\ell,\ldots,\ell) \sim (2^{1/N}-1)^{-N\ell}\frac{1}{(2^{1/N}-1)2^{(N^2-1)/2N}\sqrt{N(\pi\ell)^{N-1}}}.
  \end{align*}
  \qed
\end{theorem}
\begin{theorem}[\cite{Eger:2015}, Theorem 4]
  Let $f$ be the indicator function on $S=\set{(s_1,\ldots,s_N)\sd s_i\in\set{1,2}}$. 
  Moreover, let $\phi=\frac{\sqrt{5}-1}{2}$ and let %. Moreover, let
  $A=-\phi^{N-1}(1+\phi)^{N-1}(1+2\phi)$. 
  Define $h
  =N\Bigl(\frac{\phi}{1+3\phi+2\phi^2}\Bigr)^{N-1}$ and 
  $b_0=\frac{1}{-\phi A \sqrt{(2\pi)^{N-1}h}}$. Then
  \begin{align*}
    c_f(\ell,\ldots,\ell) \sim \phi^{-\ell N}b_0{\ell}^{(1-N)/2}.
  \end{align*}
  \qed
\end{theorem}
\begin{example}
   For the $f$ in the last theorem, the number $c_f(9,9,9)$ equals $17,899$ while the approximation formula has $18,955.30\ldots$, which amounts to a relative error of less than $6\%$. 
\end{example}
Note that $c_f$ in the last theorem is closely related to $c_{f_{\text{prod}}}$ by Theorem \ref{theorem:fprod}, which immediatley yields another asymptotic formula.

\section{Prime criteria}\label{sec:prime}
Mann and Shanks' \cite{Mann:1972}
prime criterion states that an integer $q$ is prime if and only if $m$ divides the adjusted binomial coefficients $\binom{m}{q-2m}$ for all $m$ with $0\le 2m\le q$. This criterion can be extended to $f$-weighted integer compositions ($N=1$) when $f$ takes on the value $1$ for all elements inside the `unit sphere', that is, 0 and 1 \cite{Eger:2014,Eger:2015}. For $N\ge 1$, it is tempting to conjecture as follows.  
\begin{conjecture}\label{conjecture:mann}
  Let $f(\mathbf{x})=1$ for all $\mathbf{x}\in U_{\mathbf{0}}=\set{\mathbf{s}\in\nn^N\sd s_i\in\set{0,1}}$. Then, an integer $q>1$ is prime if and only if
  $m$ divides $\binom{m}{q\mathbf{1}-2m\mathbf{1}}_f$ for all integers $m$ with $0\le 2m\le
  q$.
\end{conjecture}
If $q$ is prime, then indeed $\binom{m}{q\mathbf{1}-2m\mathbf{1}}_f\equiv 0\pmod{m}$ for all integers $m$ with $0\le 2m\le
  q$. This is a simple consequence of Theorem
  \ref{theorem:divis}. Conversely, when $q$ is not prime, then $q$ is
  odd or even. When $q$ is even, $m=q/2$ does not divide
  $\binom{m}{q\mathbf{1}-2m\mathbf{1}}_f=\binom{m}{\mathbf{0}}_f=f(\mathbf{0})^{q/2}=1$. However,
  when $q$ is odd, the situation is more difficult. Mann and Shanks
  choose $m=(q-p)/2=pn$, for a prime divisor $p$ of $q$ and a suitable
  $n$. This choice is appropriate for $N=1$ and the stated
  requirements on $f$. However, already for
  $N=2$, we find a counter-example to this choice (when $f$ is the
  indicator function on $\set{(0,0),(0,1),(1,0),(1,1)}$). Namely, when
  $q=55$, then $p=5$ is a prime divisor of $q$ and we have $m=pn$
  where $n=5$. Then $\binom{pn}{p\mathbf{1}}_f\equiv
  \binom{pn}{p}+\frac{(pn)!}{(p!)^2(p(n-2))!}\pmod{pn}$ by Theorem
  \ref{theorem:ms} and Example \ref{example:ms}. Numerical evaluation
  shows that this sum is $\equiv 5+20\equiv 0\pmod{pn}$. However,
  while this choice is not suitable, there are others for $q=55$
  (namely $m=20,22$). We leave Conjecture \ref{conjecture:mann} as an open
  problem.   

\section{Conclusion}
Many extensions of our results are conceivable. We have shown that the basis for weighted vector compositions are sums of independent and identically distributed random vectors. Other types of compositions can be investigated in which part sizes are correlated \cite{Carlitz:1976}. The basis for this class of compositions would be sums of dependent random vectors. Many approximations both for dependent and independent sums of random variables are known, e.g., \cite{Hoeffding:1963}. How do these translate to approximation results for weighted compositions? Finally, we have generalized weighted \emph{integer} compositions to weighted \emph{vector} compositions. One could further generalize to weighted \emph{matrix} compositions or general weighted \emph{tensor} compositions, that is, compositions of arbitrary multidimensional arrays.

\section*{Acknowledgements}
The author wishes to thank the anonymous referee for helpful suggestions. 

\bibliography{lit}{}
\bibliographystyle{abbrv}

\end{document}